\documentclass{amsart}

\usepackage[foot]{amsaddr}
\usepackage{amsthm}
\usepackage{amsmath}
\usepackage{graphicx}
\usepackage[margin=1in]{geometry}
\usepackage{amssymb,amsfonts}
\usepackage[all,arc]{xy}
\usepackage{enumerate}
\usepackage{mathrsfs}
\usepackage{bbm}
\usepackage{verbatim}

\usepackage{caption}
\usepackage{subcaption}

\usepackage{tikz}

\numberwithin{equation}{section}

\DeclareMathOperator{\cc}{c}
\DeclareMathOperator{\tr}{tr}

\DeclareMathOperator{\dist}{dist}
\DeclareMathOperator{\sign}{sign}
\newcommand{\rr}{\ensuremath{\mathbb{R}}}
\newcommand{\zz}{\ensuremath{\mathbb{Z}}}
\newcommand{\ep}{\ensuremath{\varepsilon}}
\newcommand{\bdry}{\ensuremath{\partial}}

\newcommand{\tG}{\ensuremath{\tilde{G}}}

\newcommand{\norm}[2]{\ensuremath{|#1|_{#2}}}

\newcommand{\semi}[2]{\ensuremath{[#1]_{#2}}}
\newcommand{\linfty}[2]{\ensuremath{||#1||_{L^{\infty}(#2)}}}

\newcommand{\tu}{\ensuremath{\tilde{u}}}

\newcommand{\bn}{\ensuremath{\bar{n}}}

\newtheorem{thm}{Theorem}[section]

\newtheorem{cor}{Corollary}[section]
\newtheorem{prop}{Proposition}[section]
\newtheorem{lem}{Lemma}[section]

\theoremstyle{definition}

\newtheorem{defn}[thm]{Definition}

\theoremstyle{remark}
\newtheorem{rem}[thm]{Remark}

\begin{document}
\title{On a model of a population with variable motility}
\author{Olga Turanova}
\address{Department of Mathematics, University of Chicago, 5734 S. University Avenue, Chicago, IL 60637}
\email{turanova@math.uchicago.edu}
\date{March 17, 2015}

\begin{abstract}
We study  a reaction-diffusion equation with a  nonlocal  reaction term that models a population with variable motility.  We establish a global supremum bound for solutions of the equation. 
We investigate the asymptotic (long-time and long-range) behavior of the population. We perform a certain rescaling and prove that solutions of the rescaled problem converge locally uniformly to zero in a certain region and stay positive (in some sense) in another region. These regions are determined by two viscosity solutions of a related Hamilton-Jacobi equation. 
\end{abstract}

\maketitle

\section{Setting and main results}
We study  a  reaction-diffusion equation with a nonlocal reaction term:
\begin{equation}
\tag{$E$}
\label{n}
\begin{cases}
\partial_t n= \theta \partial_{xx}^2 n+ \alpha \partial_{\theta \theta}^2n+r n(1-\rho) \text{ for }(x,\theta,t)\in \rr\times \Theta\times (0,\infty),\\
\rho(x,t)=\int_\Theta n(x,\theta,t)\, d\theta \text{ for }(x,t)\in \rr\times  (0,\infty), \\
\partial_\theta n(x,\theta_{m},t)=\partial_\theta n(x,\theta_{M},t)=0 \text{ for  }(x,t)\in \rr\times (0,\infty), \\
n(x,\theta,0)=n_0(x,\theta) \text{ for  }(x,\theta)\in \rr\times \Theta.
\end{cases}
\end{equation}
This problem, introduced by B\'{e}nichou, Calvez, Meunier, and Voituriez \cite{BCMV},  models a  population  structured by a space variable $x$ and a motility trait  $\theta\in \Theta$.  Our analysis is focused on the case where the trait space $\Theta$ is a bounded subset of $(0,\infty)$.  The parameters $\alpha$ and $r$ are positive and represent, respectively, the rate of mutation and the net reproduction rate in the absence of competition. 
The problem  (\ref{n}) is of  Fisher-KPP type. The classical Fisher-KPP equation, 
\begin{equation}
\tag{F-KPP}
\label{KPP}
\partial_t m(x,t) = \beta \partial_{xx}^2m(x,t)  +r m(x,t) (1-m(x,t) ) \text{ for } (x,t)\in\rr\times (0,\infty),
\end{equation}
describes the growth and spread of a population structured by a space variable $x$.  The diffusion coefficient $\beta$ is a positive  constant. The  behavior of solutions to (\ref{KPP}) has been widely studied, starting from its introduction in \cite{Fisher, KPP}. 
The most important difference between (\ref{n}) and the  Fisher-KPP equation is that the reaction term  in (\ref{n}) is nonlocal in the trait variable. This is because competition for resources, which is represented by the reaction term, occurs between individuals of \emph{all} traits that are present in a certain  location. 
 Another  key feature of (\ref{n}) is that the trait $\theta$  affects how fast an individual moves --  this is why the coefficient of the spacial diffusion in (\ref{n}) depends on $\theta$. In addition, the trait is subject to mutation, which is modeled by the diffusion term in $\theta$. Thus, (\ref{n}) describes the interaction between dispersion of a population and the evolution of the motility trait.  We further discuss the biological motivation for (\ref{n}) and review the relevant literature in more detail later on in the introduction.

\subsubsection*{Main results}
Throughout our paper,  we assume that the trait space $\Theta$ is a bounded interval
\[
\Theta =(\theta_{m}, \theta_{M}),
\]
 where $\theta_M$ and $\theta_{m}$ are positive constants.  We study  classical solutions of (\ref{n})  with initial condition that is non-negative and ``regular enough." 
We state these    assumptions precisely as (A\ref{asump C2}), (A\ref{asm n0}) and (A\ref{asm Theta})  in Section \ref{sec:asm}. 

A significant challenge for us is that (\ref{n})  does not enjoy the maximum principle. This is due to the presence of the nonlocal reaction term. Nevertheless, we are able to establish a global upper bound for solutions of  (\ref{n}). We prove:
\begin{thm}
\label{sup bd intro}
Suppose $n$  is nonnegative, twice differentiable on $\rr\times \Theta\times (0,\infty)$ and satisfies (\ref{n}) in the classical sense, with initial condition  $n_0$ that satisfies (A\ref{asm n0}). 
There exists a constant $C$ such that
\[
\sup_{\rr\times\Theta\times(0,\infty)} n \leq C.
\]
\end{thm}

Theorem \ref{sup bd intro} is a key element in the proof of our second main result. Previous work on the Fisher-KPP equation \cite{EvansSouganidis, Friedlin} and formal computations concerning (\ref{n}) \cite{variablemotility} suggest that solutions to (\ref{n}) should converge to a travelling front in $x$ and $t$. In order to study the motion of the front, we perform the rescaling $(x,t)\mapsto (\frac{x}{\ep},\frac{t}{\ep})$.  This rescaling leads us to consider solutions $n^\ep$ of the $\ep$-dependent problem,
\begin{equation}
\tag{$E_\ep$}
\label{nep}
\begin{cases}
\ep \partial_t n^\ep = \ep^2\theta \partial^2_{xx} n^\ep+\alpha \partial_{\theta \theta} ^2n^{\ep}+rn^\ep(1-\rho^\ep) \text{ for }(x,\theta,t)\in \rr\times \Theta\times (0,\infty), \\
\rho^\ep(x,t)=\int_\Theta n^\ep(x,\theta,t)\, d\theta,  \text{ for }(x,t)\in \rr\times (0,\infty),\\
\partial_\theta n^\ep(x,\theta_{m},t)=\partial_\theta n^\ep(x,\theta_{M},t)=0 \text{ for all }(x,t)\in \rr\times (0,\infty), \\
n^\ep(x,\theta,0)=n_0(x,\theta)  \text{ for all }(x,\theta)\in \rr\times \Theta .
\end{cases}
\end{equation} 
(We point out that (\ref{nep}) is \emph{not} the rescaled version of the $\ep$-independent problem (\ref{n}), as we consider initial data $n_0(x,\theta)$, instead of $n_0(x/\ep,\theta)$, in (\ref{nep}). Please see Remark \ref{remark: relationship} for more about this.)

We study the limit of the $n^\ep$ as $\ep\rightarrow 0$. We find that there exists a set on which the sequence $n^\ep$ converges locally uniformly to zero, and another set on which a certain limit of $\int n^\ep \,d\theta $ stays strictly positive. These sets are determined by two viscosity solutions of the Hamilton-Jacobi equation, 
\begin{equation}
\tag{HJ}
\label{HJ}
\max\{u, \partial_tu-H(\partial_x u)\}=0.
\end{equation} 
The function  $H:\rr\rightarrow \rr$ arises from the eigenvalue problem (\ref{spectral})  and is determined  by  $\theta_m$, $\theta_M$, $r$, and $\alpha$ (see Proposition \ref{prop:spectral}). One may view the function $H$ as encoding the effect  of the motility trait on the limiting behavior of the $n^\ep$.  Our main result, Theorem (\ref{result on n}), says that  the Hamilton-Jacobi equation (\ref{HJ}) describes the motion of an interface which separates areas with and without individuals.

Viscosity solutions of (\ref{HJ}) with infinite initial data play a key role in our analysis and we provide a short appendix where we discuss the relevant known results. For the purposes of the introduction, we state the following lemma:
\begin{lem}
\label{lem intro}
For any $\Omega\subset\rr$, there exists a unique continuous function $u_{\Omega}$ that is a viscosity solution of (\ref{HJ}) in $\rr\times (0,\infty)$ and satisfies the initial condition
\begin{equation}
\label{infinite}
u_\Omega(x,0)= 
\begin{cases}
0 &\text{ for } x\in \Omega\\
-\infty &\text{ for } x \in \rr\setminus \bar{\Omega}.
\end{cases}
\end{equation}
In addition, we have $u_\Omega(x,t)\leq 0$ for all $x\in \rr$ and $t\in (0,\infty)$.
\end{lem}
We are interested in $u_\Omega$ for  two sets $\Omega$ determined by the initial data $n_0$. We define these two sets, $J$ and $K$, by
\begin{equation}
\label{JK}
\begin{split}
&J = \{ x\in \rr:\   \text{ there exists }\theta\in \Theta \text{ such that } n_0(x,\theta)>0\} \\
&\text{ and } \\
&K = \{ x\in \rr:\   n_0(x,\theta)>0 \text{ for all }\theta\in \bar{\Theta}\}.
\end{split}
\end{equation}
We see that $x$ belongs to $J$ if initially there is at least \emph{some} individual living at $J$, and $x$ belongs to $K$ if individuals with \emph{all} traits are present at $x$. 

Our main result says that the limiting behavior of the $n^\ep$ is determined by $u_J$ and $u_K$:
\begin{thm}
 \label{result on n}
Assume   (A\ref{asump C2}), (A\ref{asm n0}) and (A\ref{asm Theta}). Let $u_J$ and $u_K$ be the functions  given by Lemma \ref{lem intro}. Then,
\[
 \lim_{\ep\rightarrow 0} n^\ep =0\text{ uniformly on compact subsets of } \{u_J<0\} \times \Theta
\]
and
\begin{align*}
 \limsup_{\ep\rightarrow 0}{ }^* \rho^\ep(x,t)&= \lim_{\ep\rightarrow 0} \sup \{\rho^{\ep'}(y,s):\,  \ep'\leq \ep,  |y-x|,  |t-s|\leq \ep \} \\&
\geq 1  \text{ on the interior of }\{u_K=0\}.
\end{align*}
\end{thm}
Let us remark on a special case of Theorem \ref{result on n}. Suppose the initial data $n_0$ is such that the two sets $J$  and $K$ are equal (this occurs if, for example, $n_0$ is independent of $\theta$). In this case, $u_J=u_K\leq 0$ and  so $\rr = \{u_J<0\} \cup \{u_K=0\}$, which means  that Theorem \ref{result on n} gives information about the limiting behavior of $n^\ep$ \emph{almost everywhere} on $\rr$, for all times $t$.

We present the following corollary: 
\begin{cor}
\label{corc*informal}
Assume   (A\ref{asump C2}), (A\ref{asm n0}) and (A\ref{asm Theta}). Let us also suppose that each of $J$ and $K$ is a bounded interval. There exists a positive constant $c^*$, which depends only on $\alpha$, $r$, $\theta_m$ and $\theta_M$, such that 
\begin{itemize}
\item if  $\dist(x, J)>tc^*$, then $\displaystyle \lim_{\ep\rightarrow 0} n^\ep (x,\theta,t)=0$ for all $\theta\in \Theta$; and,
\item if $\dist(x, K)<tc^*$, then $\displaystyle \limsup_{\ep\rightarrow 0}{ }^* \rho^\ep(x,t)\geq 1$.
\end{itemize}
\end{cor}
The proof of Corollary \ref{corc*informal} is in Subsection \ref{subsect:pf of cors}.  It uses the work of  Majda and Souganidis \cite{MS} concerning a class of equations of the form (\ref{HJ}) but with more general Hamiltonians.
\begin{rem}
\label{remarkBC}
Corollary \ref{corc*informal} directly connects our result to the main result of Bouin and Calvez  \cite{BouinCalvez}. Indeed, Theorem 3 of \cite{BouinCalvez}  says that there exists a travelling wave solution of (\ref{n}) of  speed $c^*$, for the same $c^*$ as in Corollary \ref{corc*informal}. While it is not known whether solutions of (\ref{n}) converge to a travelling wave,  Corollary \ref{corc*informal} is a result in this direction -- it says that, in the limit as $\ep\rightarrow 0$, the  regions where the $n^\ep$ is positive and zero travel with speed $c^*$.
\end{rem}

\subsubsection*{Biological interpretation of Theorem \ref{result on n}.} 
The biological question is, as time goes on, which territory will be occupied by the species and which will be left empty? To answer this question, it is enough to  determine where the functions $u_J$ and $u_K$   are zero. In fact, Corollary \ref{corc*informal}  gives information about the limit of the $n^\ep$ simply in terms of the  sets $J$, $K$,  and a constant $c^*$. 
Indeed, we see that at time $t$, if we  stand a point that is ``far" from $J$, then there are no individuals at $x$. On the other hand, if we 
stand at a point $x$ that is ``pretty close" to $K$, then there are some  individuals living near $x$.

We formulate another corollary:
\begin{cor}
\label{cor:onn}
Assume   (A\ref{asump C2}), (A\ref{asm n0}), (A\ref{asm Theta}) and that  $J$ and $K$ are bounded intervals. Then:
\begin{itemize}
\item if  $\dist(x, J)>2t\sqrt{\theta_M r}$, then $\displaystyle \lim_{\ep\rightarrow 0} n^\ep (x,\theta,t)=0$ for all $\theta\in \Theta$; and,
\item if $\dist(x, K)<2t\sqrt{\theta_m r}$, then $\displaystyle \limsup_{\ep\rightarrow 0}{ }^* \rho^\ep(x,t)\geq 1$.
\end{itemize}
\end{cor}
We remind the reader of the fact, due to Aronson and Weinberger \cite{AW},  that $2\sqrt{\beta r}$ is the asymptotic speed of propagation of fronts for (\ref{KPP}). Thus, a consequence of Corollary \ref{cor:onn} is that, in the limit, the population we're considering spreads slower than one with constant motility $\theta_M$ and faster than one with constant motility $\theta_m$.  We give the proof of Corollary \ref{cor:onn} in Subsection \ref{subsect:pf of cors}. In addition, please see Remark \ref{rem:interpret} for  further comments on the biological implications of our results.

\subsubsection*{Biological motivation}
Biologists are interested in  the interplay  between  traits present in a species and how the species interacts with its environment -- in other words, between evolution and ecology \cite{ST, SBP, Ronce, KL-S}. It has been observed, for example in  butterflies in Britain \cite{TBWS}, that an expansion of the territory that a species occupies may coincide with changes in a certain trait  -- the butterflies that spread to new territory were able to lay eggs on a larger  variety of plants than the butterflies in previous generations. The phenotypical trait in this case  is related to  adaptation to a fragmented habitat. 

Some biologists have focused specifically on the interaction between ecology and traits that affect motility.  Phillips et al \cite{Phillips} recently discovered a  species of cane toads whose territory has, over the past 70 years,  spread with a speed that  \emph{increases}  in time. This is very interesting because this is contrary to what is predicted by the Fisher-KPP equation \cite{Fisher, KPP, AW} and has previously been observed empirically \cite{Skellam}.  
 Spacial sorting was also observed --  the toads that arrive first in the new areas have longer legs than those in the areas that have been occupied for a long time.   In addition, it was discovered that toads with longer legs are able to travel further than toads with shorter legs.  It is hypothesised that the presence of this trait -- length of legs --   is responsible for both the front acceleration and the spacial sorting. 
 Similar phenomena were observed in crickets in Britain over a shorter time period \cite{TBWS, ST}. In that case, the motility trait was wingspan. 

The cases we describe  demonstrate the need to understand the influence of a  trait -- in particular, a motility trait --  on the dynamics of a population.

\subsubsection*{Literature review}
The Fisher-KPP equation has been extensively studied, and we refer the reader to \cite{Fisher, KPP, Aronson, Fife, Murray} for an introduction. 
Hamilton-Jacobi equations similar to (\ref{HJ})  are known to arise in the analysis of the long-time and long-range behavior of (\ref{KPP}), other reaction-diffusion PDE, and systems of such equations -- see, for example, Friedlin \cite{Friedlin}, Evans and Souganidis \cite{EvansSouganidis} and Barles, Evans and Souganidis \cite{BarlesEvansSouganidis} and Fleming and Souganidis \cite{FS}. The methods of \cite{EvansSouganidis, BarlesEvansSouganidis, FS} are a key part of our analysis of (\ref{nep}).

As we  previously mentioned, (\ref{KPP}) describes  populations structured by space alone, while there is a need to study the interaction of dispersion and phenotypical traits (in particular, motility traits). Most models of populations structured by space and trait either consider a trait that does not affect motility or do not consider the effect of mutations.  Champagnat and M\'{e}l\'{e}ard \cite{CM} start with an individual-based model of  such a population and derive a PDE that describes its dynamics. In the case that the trait affects only the growth rate and not the motility, Alfaro, Coville and Raoul \cite{ACR} study this PDE, which is a reaction-diffusion equation with constant diffusion coefficient:
\begin{equation}
\label{ACR}
n_t - \Delta_{x,\theta} n = \left(r(\theta - x) - \int_\rr K(\theta - x, \theta'-x)n (x, \theta', t)\,  d\theta'\right) n.
\end{equation} 
The  population modeled by (\ref{ACR}) has a preferred trait that varies in space.   Berestycki, Jin and Silvestre \cite{BJS} analyze an equation similar to (\ref{ACR}), but with a different kernel $K$ and growth term $r$ that represent the existence of a trait that is favorable for all individuals.   
The  aims and methods of \cite{ACR, BJS} are quite different from those in this paper. The main result of \cite{ACR} is the existence of traveling wave solutions of (\ref{ACR}) for speeds above a  critical threshold. In \cite{BJS}, the authors establish  the  existence and uniqueness of travelling wave solutions and prove an asymptotic speed of propagation result for the equation that they consider. 
We also mention that  a local version of  (\ref{ACR}) was investigated by Berestycki and Chapuisat  \cite{BC}.  
Desvillettes,  Ferriere and  Pr\'{e}vost   \cite{DFP} and Arnold, Desvillettes and  Pr\'{e}vost \cite{ADP} study a model in which the dispersal rate does depend on the trait and the trait is subject to mutation, but the mutations are represented by a nonlocal linear term, not a diffusive term. 

There has also been analysis of traveling waves and steady states for  equations of the form
\begin{equation}
\label{kernel}
v_t = \Delta v +v(1-v\ast\phi) \text{ in } \rr \times (0,\infty),
\end{equation}
where the reaction term $v\ast\phi$ is the convolution of $v$ with some kernel $\phi$.  We refer the reader to Berestycki, Nadin, Perthame and Ryzhik \cite{BNPR},  Hamel and Ryzhik \cite{HamelRyzhik}, Fang and Zhao \cite{FZ}, Alfaro and Coville \cite{AC} and the references therein. An important difference between (\ref{kernel}) and (\ref{n}) is that the  reaction term of (\ref{n}) is local in the space variable and nonlocal in the trait variable, while  the reaction term in (\ref{kernel}) is fully nonlocal. The long time behavior of solutions to (\ref{kernel}) is studied in \cite{HamelRyzhik}. In addition, \cite[Theorem 1.2]{HamelRyzhik} establishes a supremum bound for solutions of (\ref{kernel}). 

Bouin and Mirrahimi \cite{BouinMirrahimi} analyze the reaction-diffusion equation 
\begin{equation}
\label{eqnBM}
v_t=D\Delta v + \alpha v_{\theta \theta} + rv(x,\theta,t)\left(a(x,\theta)-\int_{\Theta}v(x,\theta,t)\,d\theta\right) \text{ in }\rr^d\times \Theta\times (0,\infty),
\end{equation}
 with Neumann  conditions on $v$ on the boundary of $\Theta$. The main difference between (\ref{eqnBM}) and (\ref{n}) is that the coefficient of spacial diffusion  in (\ref{eqnBM}) is constant, which means that (\ref{eqnBM}) models  a population where the trait does not affect motility. The methods we use here are similar to those of \cite{BouinMirrahimi} (and, in turn, both ours and those of \cite{BouinMirrahimi} are similar to those used in \cite{EvansSouganidis, BarlesEvansSouganidis, FS} to study (\ref{KPP})). 
However,  in general it is  easier to obtain certain bounds for solutions of (\ref{eqnBM}) than for solutions of (\ref{n}). For example, because the  coefficient of $\Delta v$  in (\ref{eqnBM}) is constant, integrating (\ref{eqnBM}) in $\theta$ implies that $\int_{\Theta}v(x,\theta,t)\,d\theta$ is a subsolution of a local equation in $x$ and $t$ that enjoys the maximum principle. This immediately implies that  $\int_{\Theta}v(x,\theta,t)\,d\theta$ is globally bounded \cite[Lemma 2]{BouinMirrahimi}. This strategy does not work for (\ref{n}). 
Indeed, a serious challenge in studying (\ref{n}), as opposed to (\ref{KPP}) or  nonlocal reaction diffusion equations with constant diffusion coefficient such as (\ref{eqnBM}),  is obtaining a global supremum bound for solutions of (\ref{n}). Another challenge that arises in our situation but not in \cite{BouinMirrahimi} is in establishing certain gradient estimates in $\theta$ (see Remark \ref{rem:gradbdBM}). In addition, 
 we compare our main result, Theorem \ref{result on n}, with that of \cite{BouinMirrahimi} in Remark \ref{remrho}.

 Let us discuss the literature that directly concerns (\ref{n}) and (\ref{nep}). The problem (\ref{n}) was introduced in \cite{BCMV}. The rescaling leading to (\ref{nep}) was suggested by Bouin,  Calvez, Meunier, Mirrahimi,  Perthame,  Raoul, and Voituriez  in \cite{variablemotility}.  In addition,   formal results about the asymptotic behavior of solutions to (\ref{n})  were obtained in \cite{variablemotility}. In particular,   Part (I) of Proposition \ref{main result} of our paper    was predicted  in \cite[Section 2]{variablemotility}. (We briefly remark on the term ``motility." It was used heavily in \cite{variablemotility}, which is where   we first learned of the problem (\ref{n}). However, one of the referees of our paper pointed out that this term applies mainly to unicellular organisms.)
 
Bouin and Calvez \cite{BouinCalvez}  also study (\ref{n}). They  prove that there exist  traveling wave solutions to (\ref{n}) but do not analyze whether solutions converge to a traveling wave. In fact, to our knowledge, there are no previous rigorous results about the asymptotic behavior of  solutions of (\ref{n}) or the limiting behavior of solutions to (\ref{nep}). The main difficulty is the lack of comparison principle for (\ref{n}). Please see Corollary \ref{corc*informal}, Remark \ref{remarkBC}, as well as the fourth point below, for further discussion of the connections between the results of our paper and those of \cite{BouinCalvez}.

\subsubsection*{Contribution of our work}
\begin{itemize}
\item To the best of our knowledge, Theorem \ref{sup bd intro} is the first global supremum bound for a Fisher-KPP type equation with a nonlocal reaction term and non-constant diffusion. 
\item Theorem \ref{result on n} completes the program that was proposed in \cite{variablemotility} for analyzing the asymptotic behavior of the model (\ref{n}) in the case where the trait space $\Theta$ is bounded. 
\item  We view our main result, Theorem \ref{result on n}, as evidence that the presence of a motility trait does affect the limiting behavior of populations. 
\item  Corollary \ref{corc*informal}  provides a direct connection between our work and the main result of  \cite{BouinCalvez}. Indeed,  \cite[Theorem 3]{BouinCalvez} states that there exist travelling wave solutions to (\ref{n}) of a certain speed, while Corollary \ref{corc*informal} shows that this exact speed characterizes the limiting behavior of the $n^\ep$.
\item We hope that our work  is a step towards analyzing (\ref{n}) in the case where the trait space $\Theta$ is unbounded. It is in this case that the phenomena of accelerating fronts is predicted to occur \cite{variablemotility}.
\end{itemize}

\subsubsection*{Elements of the proofs of the main results.} 
The proof of Theorem \ref{sup bd intro} is quite involved. The difficulty comes from the combination of the nonlocal reaction term and non-constant diffusion. 
Our proof of Theorem \ref{sup bd intro} uses regularity estimates for solutions of elliptic PDE, a heat kernel estimate, and an averaging technique similar to that of \cite[Theorem 1.2]{HamelRyzhik}. We believe this combination of methods is new and may be useful in other contexts. We include a detailed outline in Subsection \ref{subsec:outline}.

To analyze the limit of the $n^\ep$, we preform the transformation
\begin{equation}
\label{def:up}
u^\ep(x,\theta,t)=\ep \ln (n^\ep(x,t,\theta)).
\end{equation} 
Such a transformation is used in \cite{EvansSouganidis, BarlesEvansSouganidis, variablemotility, BouinMirrahimi}. 
We prove locally uniform estimates on $u^\ep$: 
\begin{prop}
\label{bd on utheta}
Assume  (A\ref{asump C2}), (A\ref{asm n0}) and (A\ref{asm Theta}) and let $u^\ep$ be given by (\ref{def:up}). 
Suppose $Q$ is compactly contained in $\rr\times (0,\infty)$. There exists a constant $C$ that depends on $Q$, $\alpha$, $r$, $\theta_m$ and $\theta_M$ such that for all  $0<\ep<1$ and for all $(x,\theta,t)$ such that $(x,t)\in Q$ and $\theta \in \Theta$, we have
\[
 -C\leq u^\ep(x,\theta,t)\leq \ep \ln C
\]
and
\[
|u^\ep_\theta (x,\theta, t)|\leq \ep^{1/2} C.
\]
\end{prop}
We define the half-relaxed limits $\bar{u}$ and $\underline{u}$ of $u^\ep$ by,
\begin{equation}
\label{halfrelaxed}
\begin{split}
&\bar{u}(x,t)=\lim_{\ep\rightarrow 0} \sup \{u^{\ep'}(y,\theta,s):\,  \ep'\leq \ep,  |y-x|,  |t-s|\leq \ep, \theta\in \Theta \} \\
&\text{ and }\\
&\underline{u}(x,t)=\lim_{\ep\rightarrow 0}\inf\{u^{\ep'}(y,\theta,s):\, \ep'\leq \ep, |y-x|,  |t-s|\leq \ep, \theta\in \Theta\}.
\end{split}
\end{equation}
The supremum estimates of Proposition \ref{bd on utheta} imply  that $\underline{u}$ and $\bar{u}$ are finite everywhere on $\rr\times(0,\infty)$. Moreover, the gradient estimate of Proposition \ref{bd on utheta} implies  $u^\ep$ becomes independent of $\theta$ as $\ep$ approaches zero. Thus, it is natural that  $\bar{u}$ and $\underline{u}$ should be independent of $\theta$. There is also a connection to homogenization theory -- we can think of $\theta$ as the ``fast" variable, which disappears in the limit. It is the Hamilton-Jacobi equation (\ref{HJ}), which arises in the $\ep\rightarrow 0$ limit, that captures the effect of the ``fast" variable.

We use  a perturbed test function  argument (Evans \cite{EvansPTF}) and  techniques similar to the proofs of \cite[Theorem 1.1]{EvansSouganidis}, \cite[Propositions 3.1 and 3.2]{BarlesEvansSouganidis}, and \cite[Proposition 1]{BouinMirrahimi} to establish:
\begin{prop}
\label{main result}
Assume   (A\ref{asump C2}), (A\ref{asm n0}) and (A\ref{asm Theta}) and let $\bar{u}$ and $\underline{u}$ be given by (\ref{halfrelaxed}).  Then:
\begin{description}
\item[\bf{(I)}] $\bar{u}$ is a viscosity subsolution and $\underline{u}$ is a viscosity supersolution of (\ref{HJ})
in $\rr\times  (0,\infty)$; and
\item[\bf{(II)}] we have 
\begin{equation}
\label{eqn:us at time 0}
\bar{u}(x,0)= 
\begin{cases}
0 &\text{ for } x\in J\\
-\infty &\text{ for } x \in \rr\setminus \bar{J}
\end{cases}
\end{equation}
and
\begin{equation}
\label{eqn:u under at 0}
 \underline{u} (x,0)= 
\begin{cases}
0 &\text{ for } x\in K\\
-\infty &\text{ for } x \in \rr\setminus K.
\end{cases}
\end{equation}
\end{description}
\end{prop}
Part (I) of Proposition \ref{main result} was predicted via formal arguments in \cite[Section 2]{variablemotility}. 

 Theorem \ref{result on n} follows easily from Proposition \ref{main result} by arguments  similar to those in the proofs of \cite[Theorem 1.1]{EvansSouganidis} and \cite[Theorem 1]{BarlesEvansSouganidis}.

\begin{rem}
\label{remrho}
An interesting question is whether Theorem \ref{result on n} can be refined to obtain better information about the limit of the $\rho^\ep$ in the interior of the set $\{u_K=0\}$. For instance, is   $\liminf_* \rho^\ep$ bounded from below in this set? 
When the  diffusion coefficient is constant, which is the situation  studied in \cite{BouinMirrahimi}, the answer is yes. Indeed, \cite[Theorem 1]{BouinMirrahimi} provides a lower bound on $\liminf_{\ep\rightarrow 0}\int_{\Theta}v^\ep(x,\theta,t)\,d\theta$, where $v^\ep$ is a rescaling of the solution $v$ of (\ref{eqnBM}). This lower bound of \cite[Theorem 1]{BouinMirrahimi} is obtained using an argument that relies on the diffusion coefficient of $\Delta v$ being constant, and thus does not work in our case. 
\end{rem}

\subsubsection*{Structure of our paper} 
In the next  subsection, we state our assumptions, give notation, and provide the definition of $H$.  
 The rather lengthy  Section \ref{sec:sup bd} is  devoted to the proof of Theorem \ref{sup bd intro}. This section is self-contained. In Section \ref{sec:bdutheta} we prove Proposition \ref{bd on utheta}. The proof of Proposition \ref{main result} is in Section \ref{sec:limitsuep}. The proof of  Theorem \ref{result on n} is given in Section \ref{sec:result on n}, and the proofs of Corollaries \ref{corc*informal} and \ref{cor:onn} are in subsection \ref{subsect:pf of cors}. We also provide an appendix with a discussion of  results on existence, uniqueness, and comparison for Hamilton-Jacobi equations with infinite initial data.

We have organized our paper so that a reader who is interested mainly in our  proof of Theorem \ref{sup bd intro} may only read Section \ref{sec:sup bd}. On the other hand,  a reader who is interested in our results about the limit of the $n^\ep$, and not in the proofs of the supremum bound on $n^\ep$ and the estimates on $u^\ep$, may skip ahead to Sections  \ref{sec:limitsuep} and \ref{sec:result on n} after finishing the introduction.

\tableofcontents

\subsection{Ingredients}
\subsubsection{Assumptions}
\label{sec:asm}
Our results hold under the following assumptions: 
\begin{enumerate}[({A}1)]
\item \label{asump C2} $n^\ep$ is a non-negative classical solution of  (\ref{nep}).
\item \label{asm n0} 
$n_0(x,a(\theta))\in C^{2,\eta}(\rr\times \rr)$ for some $\eta\in (0,1)$,  where $a(\theta):\rr\rightarrow [\theta_m,\theta_M]$ is defined  by (\ref{a}).
\item  \label{asm Theta}  $\Theta = (\theta_m, \theta_M)$, where $0<\theta_m<\theta_M$.
\end{enumerate}
%
The assumption  (A\ref{asm n0}) implies that $n_0(x,\theta)$ is contained in  $C^{2,\eta}(\rr\times \Theta)$ and satisfies 
\[
\partial_\theta n_0(x,\theta_m)=\partial_\theta n_0(x,\theta_M)=0
\]
 for all $x\in \rr$. We use  condition (A\ref{asm n0}) in our proof of Theorem \ref{sup bd intro}. 

  
\begin{rem}
\label{remark: relationship}
We remark on the relationship between (\ref{n}) and (\ref{nep}). 
Our results in Theorem \ref{result on n} concern the limiting behavior of $n^\ep$, and we emphasize that we do not make any precise statements concerning the asymptotic behavior of  solutions to (\ref{n}). As we mentioned ealier, this is due to our assumption that the initial data of (\ref{nep}) is independent of $\ep$. 

Indeed, let us  suppose $n$ satisfies (\ref{n}) and \emph{define} $n^\ep(x, \theta,t)=n(\frac{x}{\ep},\theta,\frac{t}{\ep})$. We find that $n^\ep$  satisfies (\ref{nep}), but with initial data $n_0(\frac{x}{\ep},\theta)$. Thus,  our assumption that the initial data of (\ref{nep}) does not depend on $\ep$ is essentially saying that we study the problem (\ref{n}) with  initial data that is invariant under the rescaling $x\mapsto \frac{x}{\ep}$. One example of such data is $\chi_{(-\infty,0]}(x)f(\theta)$, where $\chi_{(-\infty, 0]}$ is the indicator function of $(-\infty, 0]$. However, this example does not satisfy the regularity assumption (A\ref{asm n0}). We leave for future work the possibility of establishing a supremum bound on $n$ that does not require such regular initial data. 

Let us also explain why we consider the problem (\ref{nep}) with initial data independent of $\ep$.
A key element of our study is the transformation $u^\ep(x,\theta,t)=\ep \ln (n^\ep(x,t,\theta))$.   We know that $n^\ep\geq 0$, but  it is not necessarily true that $n^\ep>0$. It is therefore possible that $u^\ep$,  $\underline{u}$, or $\overline{u}$ ``take on the value $-\infty$". In order for our analysis to make sense, we need to eliminate this possibility  for times $t>0$.

To do this, we can make one of two assumptions regarding the initial data $n^\ep_0$. 
One option is to assume that $n_0^\ep>0$ holds everywhere (this  is essentially the assumption made in \cite[line (1.5)]{BouinMirrahimi}). 
Another option is to assume that the initial data  does not depend on $\ep$. This is the assumption made in \cite{EvansSouganidis, BarlesEvansSouganidis}, as well as in this paper. We  find that  $u^\ep$ is bounded from below on compact sets (see Proposition \ref{bd on utheta}). Such a bound is proven by a barrier argument  similar to those in \cite{EvansSouganidis}, but this argument does not work when the initial data depends on $\ep$. 
\end{rem}

\subsubsection{Notation} 
\label{notation}
We will slightly abuse  notation in the following way. If $Q$ is a subset of $\rr\times (0,\infty)$, then we will use  $Q\times \Theta$ to denote the set of $(x,\theta, t)$ such that $(x,t)\in Q$ and $\theta\in \Theta$. We record this as:
\[
Q\times \Theta = \{(x,\theta,t):\  \ (x,t)\in Q \text{ and } \theta\in \Theta\}.
\]

\subsubsection{The spectral problem}
Next we state the spectral problem (\ref{spectral}) of \cite[Proposition 5]{BouinCalvez}, which describes the speed $c^*$ and allows us to define the Hamiltonian $H$. 
\begin{prop}
\label{prop:spectralBC}
For all $\lambda>0$, there exists a unique solution $(\cc(\lambda), Q_\lambda(\theta))=(\cc(\lambda),Q(\theta, \lambda))$ of the spectral problem
\begin{equation}
\label{spectralBC}
\begin{cases}
(-\lambda\cc(\lambda)+ \theta \lambda^2 +r) Q(\theta,\lambda) +\alpha \partial^2_{\theta\theta}Q(\theta,\lambda) =0 \text{ for }\theta\in\Theta, &\\
\partial_\theta Q(\theta_m,\lambda)=\partial_\theta Q(\theta_M,\lambda)=0,
&\\ Q(\theta,\lambda)> 0 \text{ for all }\theta,\lambda, &\\
\int_\Theta Q(\theta,\lambda) \, d\theta =1.&
\end{cases}
\end{equation} 
The map $\lambda\mapsto \cc(\lambda)$ is continuous for $\lambda>0$, and satisfies, for all $\lambda>0$,
\begin{equation}
\label{bdoncBC}
\lambda^2 \theta_m +r\leq \lambda\cc(\lambda) \leq \lambda^2 \theta_M +r.
\end{equation}
In addition, the function $\lambda\mapsto \cc(\lambda)$ achieves its minimum, which we denote $c^*$, at some $\lambda^*>0$. 
\end{prop}

In the next proposition, we define the Hamiltonian $H$ and list the properties of $H$ that we will use.

\begin{prop}
\label{prop:spectral}
We define the Hamiltonian $H:\rr\rightarrow \rr$ by
\[
H(\lambda) = 
\begin{cases}
|\lambda| \cc(|\lambda|) \text{ for }\lambda\neq 0,\\
r \text{ for } \lambda = 0.
\end{cases}
\]
For all $\lambda\in \rr$, there exists a unique solution $(H(\lambda), Q_\lambda(\theta))=(H(\lambda),Q(\theta, \lambda))$ of the spectral problem
\begin{equation}
\label{spectral}
\begin{cases}
(-H(\lambda) + \theta \lambda^2 +r) Q(\theta,\lambda) +\alpha \partial^2_{\theta\theta}Q(\theta,\lambda) =0 \text{ for }\theta\in\Theta, &\\
\partial_\theta Q(\theta_m,\lambda)=\partial_\theta Q(\theta_M,\lambda)=0,
&\\ Q(\theta,\lambda)> 0 \text{ for all }\theta,\lambda, &\\
\int_\Theta Q(\theta,\lambda) \, d\theta =1.&
\end{cases}
\end{equation} 
Moreover,  we have that the map $\lambda\mapsto H(\lambda)$ is continuous, convex, and satisfies, for all $\lambda$,
\begin{equation}
\label{bdonc}
\lambda^2 \theta_m +r\leq H(\lambda) \leq \lambda^2 \theta_M +r.
\end{equation}
We have, for all $\lambda\in \rr$,
\begin{equation}
\label{Hc*}
\inf_{\alpha>0 }\left\{\alpha H\left(\frac{\lambda}{\alpha}\right)\right\} = |\lambda| c^*,
\end{equation}
where $c^*$ is as in Proposition \ref{prop:spectralBC}. Finally, we have
\begin{equation}
\label{boundsonc*}
2\sqrt{\theta_m r}\leq c^* \leq 2\sqrt{\theta_M r} .
\end{equation}
\end{prop}

\begin{proof}[Proof of Proposition \ref{prop:spectral}]
First, let us verify that $H$ is continuous. 
Since the bound (\ref{bdoncBC}) holds for all $\lambda>0$,  we have 
\[
\lim_{\lambda \rightarrow 0 } \lambda \cc(\lambda)= r,
\]
and in particular this limit is finite, so that $\lambda\mapsto H(\lambda)$ is continuous on all of $\rr$. We also note that (\ref{bdonc}) does indeed hold for all $\lambda$. 
We remark that $H$ is the even reflection of  the map $\lambda\mapsto \lambda \cc(\lambda)$.

For $\lambda<0$, we define $Q(\theta, \lambda)$ to be $Q(\theta, |\lambda|)$, where the latter is the solution of (\ref{spectralBC}) corresponding to $|\lambda|$. If $\lambda=0$, we see that $Q\equiv (\theta_M-\theta_m)^{-1}$ solves (\ref{spectral}). Hence, (\ref{spectral}) holds for all $\lambda$.

In addition, according to the proof of Proposition \ref{spectralBC}, the map $\lambda \cc(\lambda)$ is given by,
\begin{equation}
\label{cAndGamma}
\lambda \cc(\lambda)= \lambda^2\theta_M +r - \gamma(\lambda),
\end{equation}
where $1/\gamma(\lambda)$ is the eigenvalue of a certain eigenvalue problem with parameter $\lambda$. From this characterization, one can check that  $1/\gamma(\lambda)$ is convex in $\lambda$. Hence $\lambda \cc(\lambda)$ is also convex in $\lambda$.  The even reflection of a convex function is not in general convex, due to a possible issue at 0. However, the fact that $H$ satisfies (\ref{bdonc}) allows us to deduce that $H$ is indeed convex.

Finally, we  verify (\ref{Hc*}). According to the definition of $H$ we have,
\begin{align*}
\inf_{\alpha>0}\left\{\alpha H\left(\frac{\lambda}{\alpha}\right)\right\} &=
\inf_{\alpha>0}\left\{\alpha |\frac{\lambda}{\alpha}| \cc\left(|\frac{\lambda}{\alpha}|\right) \right\} = |\lambda| \inf_{\alpha>0 }\left\{ \cc\left(|\frac{\lambda}{\alpha}|\right) \right\}.
\end{align*}
According to Proposition \ref{prop:spectralBC}, $\cc$ achieves its minimum, $c^*$, at some $\lambda^*>0$. Thus, the infimum in the previous line is achieved at $\alpha= |\lambda|/\lambda^*$. Hence we have established (\ref{Hc*}).

To verify (\ref{boundsonc*}), we use (\ref{bdoncBC}) to find,
\[
\lambda \theta_m +r\lambda^{-1}\leq \cc(\lambda) \leq \lambda \theta_M +r\lambda^{-1}.
\]
Hence the minimum value of $\cc(\lambda)$ must be between the minimum values of the functions on the left-hand side and the right-hand side. These are exactly $2\sqrt{r\theta_m}$ and  $2\sqrt{r\theta_M}$, respectively. This establishes (\ref{boundsonc*}) and completes the proof of the proposition.

\end{proof}

\begin{rem}
\label{rem:interpret}
We give a further (slightly informal) interpretation of Theorem \ref{result on n}. According to \cite[Theorem 1.1]{EvansSouganidis}, the behavior of solutions to (\ref{KPP}) is characterized by a solution $v$ of the Hamilton-Jacobi equation
\begin{equation}
\label{forKPP}
\max\{v_t-\beta \cdot (v_x)^2 -r,v\} =0.
\end{equation}
(We remark that this is the negative of the equation that appears in \cite{EvansSouganidis}; we write it this way to be consistent with the signs employed in the rest of this paper.)  
Let us compare  (\ref{forKPP}) and  (\ref{HJ}). 
We see that both  are of the form $\max\{ v_t -H(v_x),v\}$, but for different Hamiltonians $H$. It is this difference that captures the effect of the trait. Indeed, according to Proposition \ref{prop:spectral}, the term $H (u_x)$ satisfies (\ref{bdonc}). 
Thus, we see that the term $H(u_x) $ in (\ref{HJ})  is ``like a quadratic", but of different size than the quadratic in  (\ref{forKPP}). In addition, according to (\ref{cAndGamma}) and the definition of $H$, we have
\[
H(\lambda) = \lambda^2\theta_M +r-\gamma(\lambda),
\]
where $\gamma(\lambda)$ is \emph{positive} for all $\lambda$. Hence the Hamiltonian in  (\ref{HJ}) is never exactly a quadratic, and hence must be different from that of (\ref{forKPP}). 

Since (\ref{HJ}) and (\ref{forKPP}) characterize the behavior of populations with and without a motility trait, respectively,  we interpret our results as evidence that the presence of the trait does affect the asymptotic behavior of a population.
\end{rem}

\section{Supremum bound}
\label{sec:sup bd} 
This section is devoted to the proof of:
\begin{thm}
\label{prop:supbd} Suppose $n$  is nonnegative, twice differentiable on $\rr\times \Theta\times (0,\infty)$ and satisfies (\ref{n}) in the classical sense. Assume that $\Theta$ satisfies (A\ref{asm Theta}) and the initial condition  $n_0$  satisfies (A\ref{asm n0}). 
There exists a constant $C$ that depends only on $\theta_m$, $\theta_M$, $ \alpha$, $r$ and $\eta$ such that
\begin{equation}
\label{supbdconclusion}
\sup_{\rr\times\Theta\times(0,\infty)} n \leq C\max \left\{ 1, \sup_{\rr\times \Theta} n_0,  (\norm{n_0(x,a(\theta))}{2+\eta,\rr\times\rr})^{\frac{2}{\eta}}\right\}.
\end{equation}
\end{thm}

We briefly explain notation for norms and seminorms, which we will be using only in this section. For $U\subset \rr^{d+1}$ we denote: 
\[
\semi{u}{\eta,U} = \sup_{(x,t), (y,s)\in U}\frac{|u(x,t)-u(y,s)|}{(|x-y|+|s-t|^{1/2})^{\eta}}, \   \    \   \norm{u}{\eta,U}= \linfty{u}{U}+\semi{u}{\eta,U},
\]
and
\[
\norm{u}{2+\eta,U}= \linfty{u}{U}+\sum_{i=1}^d\linfty{u_{x_i}}{U} +\linfty{u_t}{U}+ \sum_{i,j=1}^{d}\norm{u_{x_ix_j}}{\eta, U} +\norm{u_t}{\eta, U} .
\]

We present the supremum bound for solutions of (\ref{nep}):
\begin{cor}
\label{supbd:cor}
Assume (A\ref{asump C2}), (A\ref{asm n0}) and (A\ref{asm Theta}). 
There exists a constant $C$ that depends only on $\theta_m$, $\theta_M$, $ \alpha$, $r$, $\eta$, and $\norm{n_0(x,a(\theta))}{2+\eta,\rr\times\rr}$ such that for all $0<\ep<1$,
\[
\sup_{\rr\times\Theta\times(0,\infty)} n^\ep \leq C.
\]
\end{cor}
Let us explain how Corollary \ref{supbd:cor} follows  from Theorem \ref{prop:supbd}.
\begin{proof}[Proof of Corollary \ref{supbd:cor}]
Let us fix some $0<\ep<1$ and suppose $n^\ep$ satisfies (\ref{nep}) in the classical sense and $n_0$ satisfies (A\ref{asm n0}). Let us define $n(x,\theta,t)=n^\ep(\ep x, \theta,\ep t)$. Then $n$ satisfies (\ref{n}) with initial data $\tilde{n}_0(x,\theta):=n_0 (\ep x, \theta)$. We have  that $\tilde{n}_0$ satisfies (A\ref{asm n0}). Therefore, according to Theorem \ref{prop:supbd}, we have that the estimate (\ref{supbdconclusion}) holds, with $\tilde{n}_0$ instead of $n_0$ on the right-hand side. Since we have $\sup_{\rr\times \Theta} \tilde{n}_0 = \sup_{\rr\times \Theta} n_0$ and $\norm{\tilde{n}_0(x,a(\theta))}{2+\eta,\rr\times\rr} \leq \norm{n_0(x,a(\theta))}{2+\eta,\rr\times\rr}$, we obtain the conclusion of Corollary \ref{supbd:cor}.
\end{proof}

\subsection{Outline}
\label{subsec:outline}
We introduce the following piecewise function $a(\theta):\rr\rightarrow [\theta_m,\theta_M]$. For any $\theta\in \rr$ let $k\in \zz$ be such that $\theta-\theta_m\in [k(\theta_M-\theta_m),(k+1)(\theta_M-\theta_m))$. We define $a(\theta)$ by
\begin{equation}
\label{a}
a(\theta)=
\begin{cases}
\theta- k(\theta_M-\theta_m) \text{  if  $k$ is even},\\
(k+1)(\theta_M-\theta_m) -\theta \text{ if $k$ is odd}.
\end{cases}
\end{equation}

\begin{figure}

\caption{On the left is a cartoon graph of the extension $\theta\mapsto a(\theta)$. On the right is a cartoon graph of the extension $\theta\mapsto \bn(x,\theta,t)$. The parts of the graphs that are in black and very thick  represent the original functions $\theta\mapsto \theta$ and $\theta\mapsto n(x,\theta,t)$ on the domain $\Theta$. }
\label{fig:aAndn}
\centering \includegraphics[trim = 3cm 23.5cm 0cm 2.2cm  , clip, scale=1]{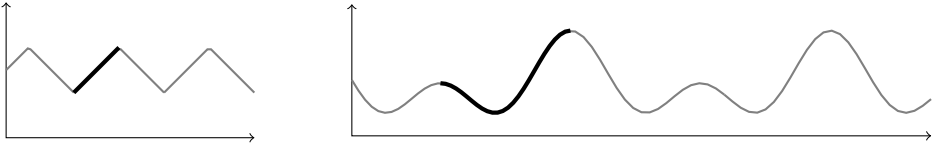}\end{figure}

We use $a$ to extend $\theta\mapsto n(x,\theta,t)$ to a function $\bar{n}(x,\theta,t)=n(x,a(\theta),t)$ defined for all $\theta\in \rr$. This is equivalent to extending $\theta\mapsto n(x,\theta,t)$ by even reflection across $\theta=\theta_M$ (since $n$ satisfies Neumann boundary conditions, this gives a periodic function on $(\theta_m, \theta_m+2(\theta_M-\theta_m))$ with derivative zero on the boundary) and then by periodicity to the rest of $\rr$.  Please see Figure \ref{fig:aAndn}. 

If we use $a$ to extend an \emph{arbitrary} twice differentiable function on $\Theta$ with derivative zero at the boundary to all of $\rr$, the result would still have continuous first derivative but would not necessarily be twice differentiable in $\theta$. However, it turns out that since $n$ solves the equation (\ref{n}), the  extension $\bar{n}$ is in $C^{2,\eta}$. Moreover, the H\"older norm of the second derivatives of $\bn$ is bounded in terms of the supremum of $n$. We state this precisely in  Proposition \ref{deriv bd in terms of M} in   Subsection \ref{subsec:Holder}.

Instead of working directly with $\bn$, we work with its averages in $\theta$ on small intervals: we define the function $v$ by
\[
v(x,\zeta,t)=\int_{\zeta-\sigma/2}^{\zeta+\sigma/2} \bn(x,\theta,t)\,d\theta.
\]
Our idea to use these averages came  from reading the proof of \cite[Theorem 1.2]{HamelRyzhik}. We have that $v$ is twice differentiable on $\rr^2\times (0,\infty)$ and  satisfies 
\[
v_t(x,\zeta,t)=\int_{\zeta-\sigma/2}^{\zeta+\sigma/2} a(\theta) \bn_{xx}(x,\theta,t)\,d\theta + \alpha v_{\zeta\zeta}(x,\zeta,t) + r v(x,\zeta,t)(1-\rho(x,t))
\]
for all $x\in \rr$, $\zeta\in \rr$, and $t\in (0,\infty)$. 

The first step in our proof of the supremum bound on $n$ is to obtain a bound  on $\sup_{x,\zeta,t}v(x,\zeta,t)$ for some $\sigma>0$. 
We then use the  bound on $v$ to establish the supremum bound on $n$ itself. In other words, we use an $L^1$ bound for $n$ to obtain an $L^\infty$ bound. For this, we formulate Proposition \ref{prop:bdfromL1} in Subsection \ref{subsec:L1Linf}. The proof of Proposition \ref{prop:bdfromL1} uses  certain
estimates for the heat kernel associated to the operator $\partial_t - a(\theta) \partial^2_{xx} -\partial^2_{\theta \theta}$. These estimates  say that this kernel is like the  kernel for the heat equation and were established by Aronson \cite{Aronson}. 


\subsection{H\"older estimates for $n$}
\label{subsec:Holder}

We will use the following  result on the solvability in $C^{2,\eta}$ of parabolic equations:
\begin{prop}
\label{classical}
Let us take $\eta \in (0,1)$ and suppose the coefficients $a^{ij}$ are uniformly elliptic, $a^{ij}\in C^{\eta}(\rr^2)$, $f\in C^{\eta}(\rr^2\times (0,T))$, and $u_0\in C^{2+\eta}(\rr^2)$. Then the initial value problem
\begin{equation*}
\begin{cases}
u_t- a^{ij}u_{x_ix_j} = f\text{ on } \rr^2\times (0,T),\\
u=u_0 \text{ on } \rr^2\times \{t=0\}
\end{cases}
\end{equation*}
has a unique solution $u\in C^{2+\eta}(\rr^2\times [0,T))$. Moreover, there exists a constant $C$ that depends on $\eta$, the ellipticity constant of $a^{ij}$, and $\norm{a^{ij}}{\eta, \rr^2}$ such that 
\begin{equation}
\label{estimateclassprop}
\norm{u}{2+\eta,\rr^2\times [0,T]}\leq C(\norm{f}{\eta, \rr^2\times [0,T]} +\norm{u}{\eta,\rr^2\times[0,T]}+\norm{u_0}{2+\eta, \rr^2}).
\end{equation}
\end{prop}
\begin{proof}
That there exists a unique solution $u\in C^{2,\eta}$ is part of the statement of Ladyzenskaja et al \cite[Chapter IV Theorem 5.1]{Ladyzenskaja}. The estimate (\ref{estimateclassprop}) follows from Krylov \cite[Theorem 9.2.2]{Krylovbook}, which  asserts that $u$ obeys the estimate,
\begin{equation}
\label{estpreclass}
\norm{u}{2+\eta,\rr^2\times [0,T]}\leq C_1(\norm{ a^{ij}u_{x_ix_j} - u_t-u}{\eta, \rr^2\times [0,T]} +\norm{u_0}{2+\eta, \rr^2}),
\end{equation}
where the constant $C_1$ depends only on $\eta$, the ellipticity constant of $a^{ij}$ and $\norm{a^{ij}}{\eta, \rr^2}$. We use the equation that $u$ satisfies to bound from above the first term on the right-hand side of (\ref{estpreclass}) and find,
\[
\norm{ a^{ij}u_{x_ix_j} - u_t-u}{\eta, \rr^2\times [0,T]} = \norm{ f-u}{\eta, \rr^2\times [0,T]}\leq \norm{f}{\eta, \rr^2\times [0,T]} +\norm{u}{\eta,\rr^2\times[0,T]}.
\]
Using the previous line to bound from above the first term on the right-hand of (\ref{estpreclass})  yields the estimate (\ref{estimateclassprop}).
\end{proof}

In addition, we need the following lemma:
\begin{lem}
\label{lem:unique}
Assume $\Theta$ satisfies (A\ref{asm Theta}).  Given $T>0$ and continuous and non-negative functions $\rho(x,t)$ and $u_0(x,\theta)$, the solution $u$ of 
\begin{equation}
\label{eqn:ulem}
\begin{cases}
u_t=\theta u_{xx}+ \alpha u_{\theta\theta} +u(1-\rho) \text{ on }\rr\times\Theta\times (0,T),\\
 u_\theta(x,\theta_m,t)=u_\theta(x,\theta_M,t)=0 \text{ for all } x\in \rr, t\in (0,T),\\
u(x,\theta,0)=u_0(x,\theta) \text{ on } \rr\times \Theta
\end{cases}
\end{equation}
is unique.
\end{lem}
We point out that here $\rho$ is a \emph{given} function and the equation (\ref{eqn:ulem}) for $u$ is \emph{local}. Because of this, the proof of the lemma is standard and we omit it. 

We are now ready to state and prove the H\"older estimate for $n$:
\begin{prop}
\label{deriv bd in terms of M}
Assume (A\ref{asump C2}), (A\ref{asm n0}), and (A\ref{asm Theta}).   
We define  the extension $\bn:\rr\times\rr\times [0,\infty)\rightarrow \rr$ by
\[
\bn(x,\theta,t)=n(x,a(\theta),t).
\]
\begin{enumerate}
\item \label{item:bneqn}
We have that $\bn$ is twice differentiable on $\rr\times \rr\times (0,\infty)$ and satisfies
\begin{equation}
\label{eqnbn}
\begin{cases}
\bn_t = a(\theta) \bn_{xx}+  \alpha \bn_{\theta \theta} +r \bn(1-\rho) \text{ on } \rr\times \rr\times (0,\infty),\\
\bn(x,\theta,0)=\bn_0(x,\theta) \text{ for all }(x,\theta)\in \rr\times \rr.
\end{cases}
\end{equation}
\item \label{item:bdbn} Let $T>0$ and assume $\sup_{\rr\times\Theta\times [0,T]} n = M\geq 1$. There exists a positive constant $\tilde{C}$ that depends on $\eta$, $ \alpha$, $r$ $\theta_m$ and $\theta_M$ so that
\[
\norm{\bn}{2+\eta, \rr\times\rr\times(0,T]} \leq\tilde{ C} (M^{2+\eta/2}+\norm{n_0(x,a(\theta))}{2+\eta, \rr\times\rr}).
\]
\end{enumerate}
\end{prop}
\begin{proof}[Proof of Proposition \ref{deriv bd in terms of M}]
We apply Proposition \ref{classical}  with  right-hand side $f=\bn(1-\rho)$,  initial condition $u_0=\bn_0$, and the matrix of diffusion coefficients being the diagonal matrix with entries $a(\theta)$ and $\alpha$. The assumption (A\ref{asm n0}) on $n_0$ says exactly that the assumption of Proposition \ref{classical} on the initial condition is satisfied. By  Proposition \ref{classical}, there exists a unique solution $\tilde{u}$ of 
\begin{equation*}
\begin{cases}
\tu_t- a(\theta) \tu_{xx}-\alpha \tu_{\theta\theta} = r \bar{n}(1-\rho) \text{ on } \rr^2\times (0,T)\\
\tu=\bn_0 \text{ on } \rr^2 \times \{t=0\}
\end{cases}
\end{equation*}
and we have the estimate
\begin{equation}
\label{estimateclass}
\norm{\tu}{2+\eta, \rr^2\times [0,T]}\leq C(\norm{r \bar{n}(1-\rho)}{\eta, \rr^2\times [0,T]} +\norm{\tu}{\eta, \rr^2\times [0,T]}+\norm{\tu_0}{2+\eta, \rr^2}),
\end{equation}
where $C$ depends on $\theta_m$, $\theta_M$ and $\alpha$.
Let us remark that the maps $\theta\mapsto r \bar{n}(x,\theta,t)(1-\rho(x,t))$ and $\theta\mapsto\bn_0(x,\theta,t)$ satisfy
\[
v(2k(\theta_M-\theta_m)+\theta)=v(\theta) \text{ for all } \theta\in \rr, k\in \zz.
\]
Together with the fact that the solution $\tu$ is  \emph{unique}, this implies $\tu$ has this symmetry as well. Thus, $\tu_\theta(x,\theta_m, t)=\tu_\theta(x,\theta_M, t)=0$ for all $x$ and for all $t>0$. Therefore, $\tu(x,\theta,t)$ with $\theta\in \Theta$ satisfies
\begin{equation*}
\begin{cases}
\tu_t=\theta \tu_{xx}+ \alpha \tu_{\theta\theta} +r \tu(1-\rho) \text{ on }\rr\times\Theta\times (0,T),\\
 \tu_\theta(x,\theta_m,t)=\tu_\theta(x,\theta_M,t)=0 \text{ for all } x\in \rr, t\in (0,T),\\
\tu(x,\theta,0)=\bn_0(x,\theta) \text{ on } \rr\times \Theta.
\end{cases}
\end{equation*}
Since $n$ also satisfies this equation, Lemma \ref{lem:unique} implies $\tu(x,\theta,t)=n(x,\theta,t)$ for all $x\in \rr$, all $\theta\in \Theta$, and all $t\in (0,T]$. Therefore, $\bn \equiv \tu$. In particular, we have that item (\ref{item:bneqn}) of the proposition holds. In addition, the estimate (\ref{estimateclass}) holds for $\bn$ and reads,
\begin{equation}
\label{firstneta}
\norm{\bn}{2+\eta, \rr\times\rr\times [0,T]} \leq C(
\norm{\bn(1-\rho)}{\eta, \rr\times\rr\times [0,T]}+\norm{\bn}{\eta,\rr\times\rr\times[0,T]}+\norm{n_0(x,a(\theta))}{2+\eta, \rr\times\rr}),
\end{equation}
where the constant $C$ depends on $\eta$, $\alpha$, $r$, $\theta_m$ and $\theta_M$. We will now establish item (\ref{item:bdbn}) of the proposition. Let us recall that we assume 
\[
\sup_{\rr\times\Theta\times [0,T]} n =M
\]
and $M\geq 1$.  We have $\sup_{\rr\times\Theta\times [0,T]} n = \sup_{\rr\times\rr\times [0,T]} \bn$. For the remainder of the proof of this proposition, we drop writing the domains in the semi-norms and norms (it is always $\rr\times\rr\times [0,T]$).
We will now bound the right-hand side of (\ref{firstneta}) from above. To this end, we first recall the definition of $|\cdot|$:
\begin{align}
\label{usedefnorm}
\norm{\bn(1-\rho)}{\eta} &=\semi{\bn(1-\rho)}{\eta}+||\bn(1-\rho)||_{\infty}.
\end{align}
For functions $f$ and  $g$ we have the elementary estimate
\[
[fg]_{\eta}\leq ||f||_{\infty}[g]_{\eta}+||g||_{\infty}[f]_{\eta}.
\]
We apply this with $f=\bar{n}$ and $g=(1-\rho)$ to obtain a bound from above on the first term of the right-hand side of (\ref{usedefnorm}):
\begin{align*}
\norm{\bn(1-\rho)}{\eta}
&\leq \semi{\bn}{\eta}||(1-\rho)||_{\infty}+\semi{1-\rho}{\eta}||\bn||_{\infty}+ ||\bn(1-\rho)||_{\infty}.
\end{align*}
In addition, the definition of $\rho$ implies $\semi{(1-\rho)}{\eta }\leq (\theta_M-\theta_m) \semi{n}{\eta}$ and $||(1-\rho)||_{\infty}\leq  (\theta_M-\theta_m) M$. We use this to estimate the right-hand side of the previous line and find 
\begin{align}
\label{estsemifirstterm}
\norm{\bn(1-\rho)}{\eta} 
&\leq 2 (\theta_M-\theta_m) M\semi{\bn}{\eta}+M+M^2.
\end{align}
Similarly we estimate the second term on the right-hand side of (\ref{firstneta}): 
\[
\norm{\bn}{\eta} = \semi{\bn}{\eta}+||\bn||_{\infty}\leq \semi{\bn}{\eta}+M.
\]
We use (\ref{estsemifirstterm}) and the previous line to bound from above the first and second terms, respectively, on the right-hand side of (\ref{firstneta}) and obtain, 
\[
\norm{\bn}{2+\eta} \leq C(
(2 (\theta_M-\theta_m) M\semi{\bn}{\eta}+M+M^2)+\semi{\bn}{\eta} +M+\norm{n_0(x,a(\theta))}{2+\eta}) .
\]
We now use  $M\geq 1$ and obtain,
\begin{equation}
\label{n2eta}
\norm{\bn}{2+\eta} \leq
 C(M\semi{\bn}{\eta} +M^2+\norm{n_0(x,a(\theta))}{2+\eta}).
\end{equation}
By interpolation estimates  for the seminorms $[\cdot]$ (for example, \cite[Theorem 8.8.1]{Krylovbook}), we have, for any $\ep>0$,
\begin{equation}
\label{semin}
\semi{\bn}{\eta}\leq C(\ep\norm{\bn}{2+\eta} + \ep^{-\eta/2}M).
\end{equation}
We use the estimate (\ref{semin}) in the right-hand side of the bound (\ref{n2eta}) for $\norm{\bn}{2+\eta}$ and obtain 
\begin{align*}
\norm{\bn}{2+\eta} &\leq C(M(\ep\norm{\bn}{2+\eta} + \ep^{-\eta/2}M) +M^2+\norm{n_0(x,a(\theta))}{2+\eta}) 
\\&= C_1M\ep\norm{\bn}{2+\eta}  + CM^2 \ep^{-\eta/2}+CM^2+C\norm{n_0(x,a(\theta))}{2+\eta}.
\end{align*}
Choosing $\ep = \frac{1}{2C_1M}$, we obtain $C_1M \ep =\frac{1}{2} $ and $M^2 \ep^{-\eta/2}= M^{2+\eta/2}$. Rearranging the above we thus obtain,
\[
 \norm{\bn}{2+\eta} \leq CM^{2+\eta/2} +CM^2 +C\norm{n_0(x,a(\theta))}{2+\eta} \leq CM^{2+\eta/2}+C\norm{n_0(x,a(\theta))}{2+\eta},
\]
where the last inequality follows since $M\geq 1$.
\end{proof}

\subsection{How an $L^1$ bound on $n$ yields an $L^\infty$ bound}
\label{subsec:L1Linf}

We will employ certain $L^\infty$ estimates on the heat kernel, which we state in the following lemma.
\begin{lem}
\label{lem:kernel}
There exists a kernel $K$ such that if $\tilde{u}_0$ is non-negative and $\tilde{u}$ is a  weak solution of
\begin{equation*}
\begin{cases}
\tu_t= a(\theta) \tu_{xx}+\alpha \tu_{\theta\theta} \text{ on }\rr\times \rr\times (0,\infty),\\
\tu(x,\theta,0)=\tu_0(x,\theta) \text{ on } \rr\times \rr,
\end{cases}
\end{equation*}
then 
\[
\tu(x,\theta,t)=\int_{-\infty}^{\infty}\int_{-\infty}^{\infty} K(t,x,y,\theta,\zeta)\tu_0(y,\zeta)\,d\zeta\,dy.
\]
Moreover, there exist  constants $c_1$ and $c_2$ that  depend only on $\alpha$, $\theta_m$ and $\theta_M$ such that
\[
c_1t^{-1}e^{\frac{-c_1((x-y)^2+(\theta-\zeta)^2)}{t}}\leq K(t,x,y,\theta,\zeta)\leq 
c_2t^{-1}e^{\frac{-c_2((x-y)^2+(\theta-\zeta)^2)}{t}}
\]
for all $x$, $y$, $\theta$,$\zeta$ in $\rr$, and for all $t>0$.
\end{lem}
\begin{proof}
We apply two theorems of Aronson \cite{Aronson}. That we have $\tu=K* \tu_0$ is the content of \cite[Theorem 11]{Aronson}. The bound on $K$ is exactly \cite[item (ii), Theorem 10]{Aronson}.  
Aronson's results apply to any parabolic equation $u_t - \tr(A(x,\theta) D^2u)=G$ in $\rr^n\times (0,\infty)$ with bounded coefficients and with $G$ regular enough. In our case, $n=2$ and $G\equiv 0$ (so, in particular, it satisfies Aronson's hypotheses) and $A(x,\theta)$ is the diagonal matrix with entries $a(\theta)$ and $\alpha$. 
In the general case, the constants $c_1$ and $c_2$  depend on the ellipticity constant of $A$ and the $L^\infty$ norm of $A$. Since these depend only on $\alpha$, $\theta_m$ and $\theta_M$, we have that $c_1$ and $c_2$ depend only on $\alpha$, $\theta_m$ and $\theta_M$ as well.
\end{proof}
For the proof of Proposition \ref{prop:bdfromL1}  we also need the following lemma. Its proof is elementary and we omit it.
\begin{lem}
\label{lem:elem}
For $a> 0$ we have $\sum_{i=1}^{\infty} e^{-a^2i^2}\leq\frac{\sqrt{\pi}}{2a}$.
\end{lem}
We now state and prove the main proposition of this subsection.
\begin{prop}
\label{prop:bdfromL1}
Suppose $\tu$ is a non-negative classical solution of 
\begin{equation*}
\begin{cases}
\tu_t= a(\theta) \tu_{xx}+ \alpha \tu_{\theta\theta} \text{ on }\rr\times \rr\times (0,\infty),\\
\tu(x,\theta,0)=\tu_0(x,\theta) \text{ on } \rr\times \rr,
\end{cases}
\end{equation*}
where $\tu_0(x,\theta)\in C(\rr\times\rr)$ satisfies, for some $0<\sigma\leq 1$ and for all  $x\in \rr$ and $\zeta\in \rr$,
\begin{equation}
\label{asm bu0}
\int_{\zeta-\sigma/2}^{\zeta+\sigma/2} \tu_0(x,\theta)\,d\theta\leq C_1.
\end{equation}
There exists a constant $C_0$ that depends only on $ \alpha$, $\theta_m$ and $\theta_M$ such that, for all $x$ and $\theta$,
\[
\tu(x,\theta,1)\leq \frac{C_0 C_1}{\sigma}.
\]
\end{prop}
\begin{proof}[Proof of Proposition \ref{prop:bdfromL1}]
According to Lemma \ref{lem:kernel}, we have, for all $x\in \rr$, all $\theta\in \rr$, and all $t>0$,
\begin{align*}
\tu(x,\theta,t)&=\int_{-\infty}^{\infty}\int_{-\infty}^{\infty} K(t,x,y,\theta,\zeta)\tu_0(y,\zeta)\,d\zeta\,dy
\\&\leq \int_{-\infty}^{\infty}\int_{-\infty}^{\infty} ct^{-1}e^{\frac{-c((x-y)^2+(\theta-\zeta)^2)}{t}}\tu_0(y,\zeta)\,d\zeta\,dy.
\end{align*}
The second inequality follows from the upper bound on the kernel $K$ given by Lemma \ref{lem:kernel}, where we write $c$ instead of $c_2$ to simplify notation. 
Let us take $t=1$ to obtain the following bound for $\tu(x,\theta,1)$:
\begin{equation}
\label{eq:bdbu1}
\begin{split}
\tu(x,\theta,1)&\leq 
c\int_{-\infty}^{\infty}\int_{-\infty}^{\infty} e^{-c((x-y)^2+(\theta-\zeta)^2)} \tu_0(y,\zeta)\,d\zeta\,dy\\
& = c\int_{-\infty}^{\infty} e^{-c(x-y)^2} \int_{-\infty}^{\infty} e^{-c\zeta^2} \tu_0(y,\zeta-\theta)\,d\zeta\,dy.
\end{split}
\end{equation}
We will now split up the integral in $\zeta$ into a sum of  integrals over  intervals of size $\sigma$. We have
\begin{align*}
\int_{-\infty}^{\infty} e^{-c\zeta^2} \tu_0(y,\zeta-\theta)\,d\zeta
&= \sum_{i=-\infty}^{\infty} \int_{i\sigma}^{(i+1)\sigma} e^{-c\zeta^2} \tu_0(y,\zeta-\theta)\,d\zeta.
\end{align*}
For $\zeta\in (i\sigma, (i+1)\sigma)$ we have 
\[
e^{-c\zeta^2}\leq 
\begin{cases} 
e^{-ci^2\sigma^2} \text{ if } i\geq 0\\
 e^{-c(i+1)^2\sigma^2} \text{ if } i< 0.
 \end{cases}
\]
We use this to bound each of the integrals in $\zeta$ and find,
\begin{align*}
\int_{-\infty}^{\infty} e^{-c\zeta^2} \tu_0(y,\zeta-\theta)\,d\zeta
&=  \sum_{i\geq 0} e^{-ci^2\sigma^2} \int_{i\sigma}^{(i+1)\sigma} \tu_0(y,\zeta-\theta)\,d\zeta + \sum_{i< 0} e^{-c(i+1)^2\sigma^2} \int_{i\sigma}^{(i+1)\sigma} \tu_0(y,\zeta-\theta)\,d\zeta.
\end{align*}
For each $i$, we can take $\zeta=i-\sigma/2$ in  assumption (\ref{asm bu0}) of this proposition to obtain , for all $i$,
\[
\int_{i\sigma}^{(i+1)\sigma} \tu_0(y,\zeta-\theta)\,d\zeta\leq C_1.
\]
Therefore,
\begin{align*}
\int_{-\infty}^{\infty} e^{-c\zeta^2} \tu_0(y,\zeta-\theta)\,d\zeta
&\leq  2 C_1\left(1+\sum_{i=1}^{\infty} e^{-ci^2\sigma^2} \right)\leq 2C_1(1+\frac{\pi}{2\sigma\sqrt{c}}),
\end{align*}
where the last inequality follows from Lemma \ref{lem:elem} applied with $a=\sqrt{c}\sigma$. We use $\sigma\leq 1$ (and so $1\leq 1/\sigma$) to bound the right-hand side of the previous line from above and obtain,
\begin{align*}
\int_{-\infty}^{\infty} e^{-c\zeta^2} \tu_0(y,\zeta-\theta)\,d\zeta
&\leq 2C_1\left(\frac{1}{\sigma}+\frac{\sqrt{pi}}{2\sigma\sqrt{c}}\right) = \frac{C_2C_1}{\sigma}.
\end{align*}
We now use this bound on the integral in $\zeta$ in the estimate (\ref{eq:bdbu1}) for $\tu(x,\theta,1)$ and obtain
\begin{align*}
\tu(x,\theta,1)&\leq 
 \frac{C_2C_1c}{\sigma}\int_{-\infty}^{\infty} e^{-c(x-y)^2} \,dy = \frac{CC_1}{\sigma},
\end{align*}
where $C$ depends only on $\theta_m$, $\theta_M$ and $\alpha$. This holds for all $x$ and $\theta$, so the proof is complete.
\end{proof}

\subsection{Proof of  the supremum bound}
We have now established the two auxillary results that we need, and are ready to proceed with:
\begin{proof}[Proof of Theorem \ref{prop:supbd}]
Let  $C_0$ be the constant from Proposition \ref{prop:bdfromL1} and  let $\tilde{C}$ be the constant from Proposition \ref{deriv bd in terms of M}. Define the constants $\bar{C}$ and $M_0$ by
\begin{equation}
\label{barC}
\bar{C} = \tilde{C}(1+ \norm{n_0(x,a(\theta))}{2+\eta, \rr\times\rr})
\end{equation}
and 
\begin{equation}
\label{choiceM0}
M_0=\max\left\{(3e^rC_0)^{\frac{4+2\eta}{\eta}}\bar{C}^{\frac{2}{\eta}}r^{-\frac{2}{\eta}}, e^r\sup_{\rr\times\Theta}n_0, \frac{3e^rC_0}{\theta_M-\theta_m},1, 3e^rC_0\right\}.
\end{equation}
We claim 
\[
\sup_{\rr\times\Theta\times (0,\infty)}n\leq M_0.
\]
 We proceed by contradiction: let us assume $\sup_{\rr\times\Theta\times (0,\infty)}n> M_0$. Let us now fix a number $M$  with $M>M_0$ and
\begin{equation}
\label{bdbn}
\sup_{\rr\times\Theta\times (0,\infty)}n>M.
\end{equation}

Throughout the rest of the proof of this proposition, $C$ denotes a positive constant that may change from line to line and depends only on  $\theta_m$, $\theta_M$, $r$, $ \alpha$ and $\eta$ (in particular, $C$ does not depend on $M$).

Let us consider the map $S(t)$ that takes $t$ to the supremum of $n$ at time $t$, in other words:
\[
S(t) =\sup_{x\in \rr,\theta\in \Theta} n (x,\theta, t).
\]
The map $S$ is continuous. In addition, since $n$ satisfies (\ref{n}), and $n$ and $\rho$ are non-negative, we have that $n(x,\theta,t)e^{-t}$ is a subsolution of
\begin{equation}
\label{eqn:ufromzero}
\begin{cases}
u_t=\theta u_{xx}+  \alpha u_{\theta\theta} \text{ on }\rr\times\Theta\times (0,\infty),\\
u(x,\theta_m,t)=u(x,\theta_M,t)=0 \text{ for all }x\in \rr, t\in (0,\infty),\\
u(x,\theta,0)=n_0(x,\theta) \text{ on } \rr\times \Theta.
\end{cases}
\end{equation}
The equation (\ref{eqn:ufromzero}) satisfies the comparison principle. Therefore, we have the following bound on $n(x,\theta,t) e^{-t}$ from above: for all $x\in \rr$ and $\theta\in \Theta$, 
\[
n(x,\theta,t)e^{-t}\leq u(x,\theta,t)\leq \sup_{x\in \rr,\theta\in \Theta} n_0(x,\theta).  
\]
Taking supremum in $x$ and $\theta$ and multiplying by $e^{t}$ gives the bound
\[
S(t)\leq e^{t}\sup_{x\in \rr,\theta\in \Theta} n_0(x,\theta)
\]
for all $t>0$. In particular, taking supremum over $t\in (0,1]$, we have
\[
\sup_{t\in (0,1]} S(t)\leq e\sup_{x\in \rr,\theta\in \Theta} n_0(x,\theta)\leq M_0<M,
\]
where the second inequality follows from the definition of $M_0$. Line (\ref{bdbn}) implies $\sup_t S(t)>M$. Since $S$ is continuous and $S(t)<M$ for $t\leq 1$, there exists a first time $T>1$ for which $S(T)=M$.
So, we have 
\[
\sup_{\rr\times\Theta\times [0,T]} n =M \text{ and }\sup_{\rr\times\Theta} n(\cdot, \cdot, T) =M. 
\]
We will now work with the extension $\bn(x,\theta,t)$ defined in Proposition \ref{deriv bd in terms of M}. By the previous line, we have
\begin{equation}
\label{supn}
\sup_{\rr\times\rr\times [0,T]} \bn =M \text{ and }\sup_{\rr\times\rr} \bn(\cdot, \cdot, T) =M.
\end{equation}
We apply Proposition \ref{deriv bd in terms of M} to $\bn$. Part \ref{item:bneqn} implies that $\bn$ satisfies equation (\ref{eqnbn}). Part (\ref{item:bdbn})   gives us the estimate
\[
\norm{\bn}{2+\eta, \rr\times\rr\times(0,T]} \leq \tilde{C}( M^{2+\eta/2}+\norm{n_0(x,a(\theta))}{2+\eta, \rr\times\rr}).
\]
Since we have $M\geq 1$,  the second term on the right-hand side of the previous line is smaller than 
\[
\norm{n_0(x,a(\theta))}{2+\eta, \rr\times\rr}M^{2+\eta/2},
\]
 so  we find,
\[
\norm{\bn}{2+\eta, \rr\times\rr\times(0,T]} 
\leq  \tilde{C}(M^{2+\eta/2}+\norm{n_0(x,a(\theta))}{2+\eta, \rr\times\rr}M^{2+\eta/2}).
\]
We use our choice of $\bar{C}$ in (\ref{barC}) to bound the right-hand side from the previous line from above and obtain,
\begin{equation}
\label{seminM}
\norm{\bn}{2+\eta, \rr\times\rr\times(0,T]} \leq  \bar{C}M^{2+\eta/2}.
\end{equation}
Let us take 
\[
\sigma=\min \{1,\theta_M-\theta_m, \bar{C}^{-\frac{1}{2+\eta}}M^{-\frac{4+\eta}{4+2\eta}}r^{\frac{1}{2+\eta}}\}
\]
 and define $v(x,\zeta,t):\rr\times\rr\times [0,\infty)\rightarrow \rr$ by
\[
v(x,\zeta,t)=\int_{\zeta-\sigma/2}^{\zeta+\sigma/2} \bn(x,\theta,t)\,d\theta.
\]

\textbf{First Step:} We will prove 
\[
\sup_{\rr\times\Theta\times (0,T)} v\leq 3.
\]

Since $n$ satisfies (\ref{n}), we have that $v$ satisfies
\begin{equation}
\label{eqnv}
v_t(x,\zeta,t)=\int_{\zeta-\sigma/2}^{\zeta+\sigma/2}a(\theta) \bn_{xx}(x,\theta,t)\,d\theta + \alpha v_{\zeta\zeta}(x,\zeta,t) + r v(x,\zeta,t)(1-\rho(x,t)).
\end{equation}

Let us explain why we may assume, without loss of generality, that the supremum of $v$ on $\rr\times\Theta\times (0,T)$ is achieved at some $(x_0,\zeta_0, t_0)$. Since $\zeta\mapsto v(x,\zeta -\theta_m,t)$ is periodic of period $2(\theta_M-\theta_m)$, and we are considering times $t$ in the bounded interval $(0,T)$, we know that there exist $\zeta_0\in \Theta$, $t_0\in [0,T]$ and a sequence $\{x_k\}_{k=1}^{\infty}$ with $v(x_k, \zeta_0, t_0)\rightarrow \left(\sup_{\rr\times\Theta\times (0,T)} v\right)$ as $k\rightarrow \infty$. 
For each $k$, we define the translated functions
\[
\bn^k_{xx}(x,t)=\bn_{xx}(x+x_k,\theta, t),
\] 
\[
\rho^k(x,t)=\rho(x+x_k,t),
\]
and
\[
v^k(x,t)=v(x+x_k,\theta, t).
\] 
We similarly define the translates of the first and second derivatives of $v$. 
We have  $\bn_{xx}\in C^{\eta}$, and according to (\ref{seminM}), 
\[
\norm{\bn_{xx}}{\eta, \rr\times\rr\times[0,T]}\leq \bar{C}M^{2+\eta/2}.
\]
Therefore, $\bn^k_{xx}$, $\rho^k$, $v^k$,  and the translates of the first and second derivatives of $v$ are uniformly bounded and uniformly equicontinuous on $\rr\times\rr\times[0,T]$. Hence, there exists a subsequence (still denoted by $k$) and functions $\rho^\infty$,  $\bn_{xx}^\infty$, and $v^\infty$ such that $\rho^k$, $\bn_{xx}^k$ and $v^k$ converge locally uniformly to $\rho^\infty$,  $\bn_{xx}^\infty$ and $v^\infty$, respectively;  the derivatives of   $v^k$ converge locally uniformly to those of  $v^\infty$; and $v^\infty(x,\zeta,t)= \int_{\zeta-\sigma/2}^{\zeta+\sigma/2} \bn^\infty(x,\theta,t)\, d\theta$.
Moreover, $v^{\infty}$ satisfies
\begin{equation*}
\partial_t v^{\infty}(x,\zeta,t)=\int_{\zeta-\sigma/2}^{\zeta+\sigma/2}a(\theta) \bn^{\infty}_{xx}(x,\theta,t)\,d\theta + \alpha v_{\zeta \zeta}^{\infty}(x,\zeta,t) + r v^{\infty}(x,\zeta,t)(1-\rho^{\infty}(x,t))
\end{equation*}
on $\rr\times\rr\times (0,T)$ and we have,  for all $x\in \rr$, all $\zeta\in \Theta$, and all $t\leq T$,
\[
v^\infty(x,\zeta,t)\leq v^\infty (0,\zeta_0, t_0) =\sup_{\rr\times\Theta\times (0,T)} v;
\]
in other words, $v^\infty$ achieves its supremum on $\rr\times \rr \times (0,T)$. We now drop the superscript $\infty$.

At the point $(x_0,\zeta_0, t_0)$ where $v$ achieves its supremum, we have 
\begin{equation}
\label{derivsv}
v_t(x_0,\zeta_0,t_0)\geq 0, \  \   \   v_{\zeta\zeta}(x_0,\zeta_0, t_0)\leq 0, 
\end{equation}
and 
\begin{equation}
\label{vxx}
0\geq v_{xx}(x_0,\zeta_0,t_0)=\int_{\zeta_0-\sigma/2}^{\zeta_0+\sigma/2} \bn_{xx}(x_0,\theta,t_0)\,d\theta .
\end{equation}
We point out that $v_{xx}$ does not appear in (\ref{eqnv}), the equation that $v$ satisfies. We will bound from above the corresponding term that does appear in (\ref{eqnv}). This term is:
\begin{equation}
\label{term}
\int_{\zeta_0-\sigma/2}^{\zeta_0+\sigma/2}a(\theta) \bn_{xx}(x_0,\theta,t_0)\,d\theta.
\end{equation}
The inequality (\ref{vxx}) implies that there exists $\theta^*\in (\zeta_0-\sigma/2,\zeta_0+\sigma/2)$ with $\bn_{xx}(x_0,\theta^*, t_0)\leq 0$. In addition, let $\tilde{\theta}\in[\zeta_0-\sigma/2,\zeta_0+\sigma/2]$ be so that 
\[
\min_{\theta\in [\zeta_0-\sigma/2,\zeta_0+\sigma/2]} a(\theta) = a(\tilde{\theta}).
\]
Since  $a$ is positive, we multiply (\ref{vxx}) by $-a(\tilde{\theta})$ and find 
\begin{equation*}
0\leq \int_{\zeta_0-\sigma/2}^{\zeta_0+\sigma/2} -a(\tilde{\theta})\bn_{xx}(x_0,\theta,t_0)\,d\theta .
\end{equation*}
Adding the term  (\ref{term}) that we're interested in to both sides of this inequality, we find
\begin{align*}
\int_{\zeta_0-\sigma/2}^{\zeta_0+\sigma/2}a(\theta) \bn_{xx}(x_0,\theta,t_0)\,d\theta &
\leq \int_{\zeta_0-\sigma/2}^{\zeta_0+\sigma/2}(a(\theta)-a(\tilde{\theta}) )\bn_{xx}(x_0,\theta,t_0)\,d\theta.
\end{align*}
Let us recall that $a(\theta)-a(\tilde{\theta})$ is non-negative on $[\zeta_0-\sigma/2,\zeta_0+\sigma/2]$. We thus use the estimate (\ref{seminM}) on the seminorm of $\bn_{xx}$ and the fact that $\bn_{xx}(x_0,\theta^*, t_0)\leq 0$ to estimate the right-hand side of the previous line from above and obtain
\begin{align*}
\int_{\zeta_0-\sigma/2}^{\zeta_0+\sigma/2}a(\theta) \bn_{xx}(x_0,\theta,t_0)\,d\theta &
\leq \int_{\zeta_0-\sigma/2}^{\zeta_0+\sigma/2}(a(\theta)-a(\tilde{\theta}))(\bn_{xx}(x_0,\theta^*,t_0) + |\theta-\theta^*|^{\eta}\bar{C}M^{2+\eta/2})\,d\theta \\
&\leq \int_{\zeta_0-\sigma/2}^{\zeta_0+\sigma/2}(a(\theta)-a(\tilde{\theta}))|\theta-\theta^*|^\eta \bar{C}M^{2+\eta/2}\,d\theta.
\end{align*}
Since $\theta^*\in (\zeta_0-\sigma/2,\zeta_0+\sigma/2)$, we have $|\theta-\theta^*|\leq \sigma$. In addition, $a$ is Lipschitz with Lipschitz constant $1$, so we have $(a(\theta)-a(\tilde{\theta}))\leq |\theta-\tilde{\theta}|$. We use these  two inequalities to bound the right-hand side of the previous line from above and find
\begin{align*}
\int_{\zeta_0-\sigma/2}^{\zeta_0+\sigma/2}a(\theta) \bn_{xx}(x_0,\theta,t_0)\,d\theta &
\leq \bar{C}\sigma^{\eta}M^{2+\eta/2}\int_{\zeta_0-\sigma/2}^{\zeta_0+\sigma/2}|\theta-\tilde{\theta}|\,d\theta\leq \frac{1}{2}\bar{C}\sigma^{2+\eta}M^{2+\eta/2}.
\end{align*}
The last inequality follows by an elementary calculus computation that relies on the fact that  $\tilde{\theta}$ is contained in $[\zeta_0-\sigma/2,\zeta_0+\sigma/2]$. 

Using this estimate together with the information (\ref{derivsv}) about the other derivatives of $v$  at $(x_0,\zeta_0,t_0)$ in the equation (\ref{eqnv}) that $v$ satisfies, we obtain
\begin{equation}
\label{vatmax}
0\leq \frac{\bar{C}}{2}\sigma^{2+\eta}M^{2+\eta/2} + r v(x_0,\zeta_0,t_0)(1-\rho(x_0,t_0)).
\end{equation}
Since $\sigma\leq \theta_M-\theta_m$ we may bound $\rho(x_0,t_0)$ from below by $\frac{v(x_0,\zeta_0,t_0)}{2}$:
\[
 v(x_0,\zeta_0,t_0)= \int_{\zeta_0-\sigma/2}^{\zeta_0+\sigma/2} \bn(x_0,\theta,t_0)\,d\theta \leq 2\int_{\theta_m}^{\theta_M} n(x_0,\theta,t_0)\,d\theta = 2\rho(x_0,t_0).
\]
We use the previous estimate to bound the right-hand side of (\ref{vatmax}) from above and obtain
\[
0\leq \frac{\bar{C}}{2} \sigma^{2+\eta}M^{2+\eta/2} + r v(x_0,\zeta_0,t_0)(1-\frac{1}{2}v(x_0, \zeta_0, t_0)).
\]
Upon rearranging we find,
\begin{equation}
\label{explanationsigma}
\frac{r}{2}v^2(x_0, \zeta_0, t_0)\leq \frac{\bar{C}}{2} \sigma^{2+\eta}M^{2+\eta/2} +r v(x_0,\zeta_0,t_0).
\end{equation}
By our choice of $\sigma$, we have $\sigma\leq \bar{C}^{\frac{-1}{2+\eta}}M^{-\frac{4+\eta}{4+2\eta}} r^{\frac{1}{2+\eta}}=\bar{C}^{\frac{-1}{2+\eta}}M^{-\frac{2+\eta/2}{2+\eta}}r^{\frac{1}{2+\eta}}$. We use this to bound the  right-hand side of the previous line and find 
\[
\frac{r}{2}v^2(x_0,\zeta_0,t_0)\leq \frac{r}{2}+rv(x_0,\zeta_0,t_0),
\]
so 
\[
v(x_0,\zeta_0,t_0)\leq \frac{1}{2}(1+\sqrt{3})\leq 3.
\]
Since  $v$ achieved its supremum on $\rr \times \rr \times (0,T)$ at $(x_0,\zeta_0,t_0)$, we conclude
\begin{equation}
\label{bdv}
\sup_{\rr\times\Theta\times (0,T)}v\leq 3.
\end{equation}

\textbf{Second Step:} We will now deduce a supremum bound on $\bn$ from the bound on $v$. Let us fix any $t^*\in ( 1,T)$. Let $u$ be the solution of 
\begin{equation*}
\begin{cases}
u_t=\theta u_{xx}+ u_{\theta\theta} \text{ on }\rr\times\rr\times (0,\infty),\\
u(x,\theta,0)=\bn(x,\theta,t^*-1) \text{ on } \rr\times\rr.
\end{cases}
\end{equation*}
We have that $\bn(x,\theta,t+t^*-1)e^{-rt}$ is a subsolution of the equation for $u$ for $t\geq 0$. Since they are equal at $t=0$, the comparison principle for the equation for $u$ implies the bound
\[
\bn(x,\theta,t+t^*-1)e^{-rt}\leq u(x,\theta,t)
\]
for all $x, \theta$, and $t\geq 0$. In particular, we evaluate the above at $t=1$ and take supremum in $x$ and $\theta$ to find
\begin{equation}
\label{bdnbyu}
\sup_{x\in\rr, \theta\in \rr} \bn(x,\theta,t^*)\leq e^r\sup_{x\in\rr, \theta\in \rr}u(x,\theta,1).
\end{equation}
We will now apply Proposition \ref{prop:bdfromL1} to $u$. The supremum bound (\ref{bdv}) on $v$  says exactly that  assumption (\ref{asm bu0}) is satisfied, with $C_1=3$.
Therefore, Proposition \ref{prop:bdfromL1} implies
\begin{align*}
\sup_{x\in \rr,\theta\in\rr} u(x,\theta,1)&\leq  \frac{3C_0}{\sigma}.
\end{align*}
Therefore, 
 we may bound $u$ by $\frac{3C_0}{\sigma}$ on the right-hand side of (\ref{bdnbyu}) and find,
\[
\sup_{x\in \rr,\theta\in \rr}\bn(x,\theta,t^*)\leq \frac{3e^rC_0}{\sigma}.
\]
This holds for any $t^*\leq T$, so in particular at $t^*=T$. According to line (\ref{supn}), we have $\sup_{x,\theta}\bn(x,\theta,T)=M$, so we obtain
\begin{equation}
\label{explanationMbound}
M\leq \frac{3e^rC_0}{\sigma}.
\end{equation}
Recall that we chose $\sigma=\min \{\theta_M-\theta_m, 1, \bar{C}^{-\frac{1}{2+\eta}} M^{-\frac{4+\eta}{4+2\eta}}r^{\frac{1}{2+\eta}}\}$. If $\sigma=\theta_M-\theta_m$, then we find $M\leq \frac{3e^rC_0}{\theta_M-\theta_m}$; if $\sigma=1$, then $M\leq 3e^rC_0$; and
if $\sigma=\bar{C}^{\frac{-1}{2+\eta}}M^{-\frac{4+\eta}{4+2\eta}}r^{\frac{1}{2+\eta}}$ we obtain
\[
M\leq 3e^rC_0\bar{C}^{\frac{1}{2+\eta}}M^{\frac{4+\eta}{4+2\eta}}r^{-\frac{1}{2+\eta}},
\]
which implies, since $1-\frac{4+\eta}{4+2\eta} = \frac{\eta}{4+2\eta}$, the following bound for $M$:
\[
M\leq  (3e^rC_0)^{\frac{4+2\eta}{\eta}}\bar{C}^{\frac{4+2\eta}{\eta(2+\eta)}}r^{-\frac{4+2\eta}{\eta(2+\eta)}}= (3e^rC_0)^{\frac{4+2\eta}{\eta}}\bar{C}^{\frac{2}{\eta}}r^{-\frac{2}{\eta}}.
\]
Therefore, 
\[
M\leq \max\left\{\frac{3e^rC_0}{\theta_M-\theta_m}, (3e^rC_0)^{\frac{4+2\eta}{\eta}}\bar{C}^{\frac{2}{\eta}}r^{-\frac{2}{\eta}}, 3e^rC_0,\right\}.
\]
But we had taken $M>M_0$. Recalling our choice of $M_0$ in line (\ref{choiceM0}) yields the desired contradiction and hence the proof is complete. 
\end{proof}

\begin{rem}
\label{remark:choice of nonlinearity}
We remark on other possible choices of the nonlinear term $n(1-\rho)$, specifically as related to the proof of Theorem \ref{prop:supbd}.  (For the purposes of this remark, we use $C$ to denote any constant that does not depend on $M$.)  The exponent $2+\eta/2$ in the right-hand side of the estimate of item (\ref{item:bdbn}) of Proposition \ref{deriv bd in terms of M} is very important. Let us  suppose the estimate read,
\[
\norm{\bn}{2+\eta, \rr\times\rr\times(0,T]} \leq C (M^{p}+1),
\]
for some $p>0$. The key difference in step one of the proof of Theorem \ref{prop:supbd} is that line (\ref{explanationsigma}) would have the term $C\sigma^{2+\eta}M^p$ instead of $C\sigma^{2+\eta}M^{2+\eta/2}$. Thus, we would need to choose $\sigma = CM^{-\frac{p}{2+\eta}}$ in order to proceed with the estimate on $v$. Then, in line (\ref{explanationMbound}) in the second step of the proof, we would obtain 
\[
M\leq \frac{C}{\sigma}=CM^{\frac{p}{2+\eta}}.
\]
Hence, to deduce a bound on $M$ we need $p$ to satisfy $p<2+\eta$. 

The value of $p$ in the estimate of item (\ref{item:bdbn}) of Proposition \ref{deriv bd in terms of M} is determined by the stucture of the nonlinear term $n(1-\rho)$. For example, analyzing the proof of Proposition \ref{deriv bd in terms of M} yields that if the nonlinear term were $n^\gamma(1-\rho)$, for some $\gamma\geq 1$, the conclusion would hold with $p=1+\gamma(1+\eta/2)$. Since our method requires $p<2+\eta$, it would work for $\gamma < \frac{1+\eta}{1+\eta/2}$ (and, indeed, $\gamma=1$ satisfies this). We leave further considerations regarding different choices of nonlinearities for future work.
\end{rem}

\section{Estimates on $u^\ep$}
\label{sec:bdutheta}
In this section we give the proof of Proposition \ref{bd on utheta}. We believe that the arguments in this section are the most technical ones of the paper.

Let us point out that $u^\ep(x,\theta,t)=\ep \ln n^\ep(x,\theta,t)$ satisfies 
\begin{equation}
\tag{$U_\ep$}
\label{eqn:uep}
\begin{cases}
\partial_t u^\ep=\ep \theta \partial^2_{xx} u^\ep +\frac{\alpha}{\ep}\partial^2_{\theta \theta}u^\ep +\theta (\partial_x u^\ep)^2 + \alpha\left(\frac{\partial_\theta u^\ep}{\ep}\right)^2 +r(1-\rho^\ep)  \text{ on }\rr\times \Theta \times (0,\infty),\\
\partial_\theta u^\ep (x,\theta_m,t)=\partial_\theta u^\ep (x,\theta_M,t)=0 \text{ for all }(x,t)\in \rr\times  (0,\infty),\\
u^\ep (x, \theta, 0)= \ep \ln (n_0(x,\theta)).
\end{cases}
\end{equation}
 In subsection \ref{subsec:lowbd} we establish the upper and lower bounds on $u^\ep$. The upper bound follows directly from Corollary \ref{supbd:cor}, while the lower bound needs a barrier argument that is quite similar to \cite[Lemma 2.1]{EvansSouganidis} and \cite[Lemma 1.2 item (i)]{BarlesEvansSouganidis}. We skip a few details in the proof because it is quite similar to the cited ones. 
  In subsection \ref{subsec:grad} we prove the gradient estimate by an application of the so-called ``Bernstein method" (see \cite[Lemma 2.1]{EvansSouganidis} and  \cite[Lemma 2.2]{BarlesEvansSouganidis}). This argument is quite lengthy and has a some rather technical parts. We provide all the details.

\subsection{Proof of upper and lower bounds of Proposition \ref{bd on utheta}}
\label{subsec:lowbd}
In order to construct the barrier necessary for the proof of the lower bound, we need a lemma. The proof of the lemma  is an elementary computation which we include  for the sake of completeness.
\begin{lem}
\label{lem:barrier}
Given  $R<1/2$ we define the function $\xi(z)$ for $z\in (-R,R)$ by,
\[
\xi(z)=\frac{1}{z^2 - R^2}.
\]
There exists a constant $C$ that depends only on $R$ such that for all $z\in (-R, R)$ and for all $0<\ep<1$,
\[
-\xi_{zz}(z)-(\xi_z(z))^2\leq C \text{ and } -\ep \xi_{zz}(z)-(\xi_z(z))^2\leq C.
\]
We remark that $z$ is negative.
\end{lem}
\begin{proof}
Computing the derivatives of $\xi$, substituting them into $-\xi_{zz}(z)-(\xi_z(z))^2$ and rearranging yields,  for $z\in (-R,R)$,
\begin{align*}
-\xi_{zz}(z)&-(\xi_z(z))^2 
=  \frac{ 2 }{(z^2-R^2)^2}\left( 1  + \frac{2z^2}{R^2-z^2}\left( 2 - \frac{1}{R^2-z^2}\right)\right).
\end{align*}
We have  $R^2-z^2 \leq R^2\leq 1/4$, so that the term in the inner-most parentheses on the right-hand side of the previous line is bounded from above by $-2$.  Consequently we find, for $z\in (-R,R)$, 
\begin{align}
\label{psieqn}
-\xi_{zz}(z)&-(\xi_z(z))^2 \leq  \frac{ 2 }{(z^2-R^2)^2}\left( 1  -2 \frac{2z^2}{R^2-z^2}\right).
\end{align}
We now consider $z$ such that $\frac{R}{\sqrt{2}}\leq |z|\leq R$, in which case we have $0\leq R^2-z^2\leq R^2 - R^2/2=R^2/2$. Therefore, 
\[
\frac{z^2}{R^2-z^2}\geq 2\frac{z^2}{R^2} \geq 1,
\]
where we have again used that $ |z|\geq \frac{R}{\sqrt{2}}$. Hence the term in the parenthesis on the right-hand side of (\ref{psieqn}) is bounded from above by $-3$. Therefore,  the right-hand side of (\ref{psieqn}) is negative for  $\frac{R}{\sqrt{2}}\leq |z|\leq R$. Thus it is left to consider $z$ such that $|z|\leq \frac{R}{\sqrt{2}}$. But in this case, we see $R^2\geq R^2-z^2\geq R^2/2$. Using this, together with the fact that the term in the parenthesis on the right-hand side of (\ref{psieqn}) is bounded from above by 1, and find,
\begin{align*}
-\xi_{zz}(z)&-(\xi_z(z))^2 \leq  \frac{ 2 }{(z^2-R^2)^2}\leq 8\frac{1}{R^4}.
\end{align*}
Since $\xi_{zz}(z)$ is non-positive for $z\in (-R,R)$, we have, for $0<\ep<1$ and $z\in (-R,R)$, 
\[
-\ep \xi_{zz}(z)\leq -\xi_{zz}(z),
\]
 so that,  
 \[
 -\ep \xi_{zz}(z)-(\xi_z(z))^2 \leq -\xi_{zz}(z)-(\xi_z(z))^2\leq \frac{8}{R^4} .
 \]
  This completes the proof of the lemma.
 \end{proof}
 
The proof of the lower bound proceeds in two steps. First we fix a cube $K_R\subset \rr\times \Theta$ on which $n^0$ is strictly positive. Thus, $u^\ep(x, \theta, 0)= \ep \ln n^0(x,\theta)$ is bounded from below on this cube. We use this, together with the barrier $\xi$, to obtain a lower bound on $u^\ep(x,\theta,t)$ for $(x,\theta)$ in a smaller cube $K_R$ and for  times $t\in [0,T]$. The second step is to use this lower bound to build a barrier for $u^\ep$ on the remainder of $\rr\times \Theta\times (0,T]$. This barrier function is infinite at $\{t=0\}$, hence  the bound we obtain does not reach the initial time $0$. 

\begin{proof}[Proof of the  lower bound of Proposition \ref{bd on utheta} ]
Due to Corollary \ref{supbd:cor}, there exists a positive constant $M$ that depends only on $\theta_m$, $\theta_M$, $\alpha$ and $r$ such that $u^\ep$ is a supersolution of
\begin{equation}
\label{eqn:uepsuper}
\partial_t u^\ep=\ep \theta \partial^2_{xx} u^\ep +\frac{\alpha}{\ep}\partial^2_{\theta \theta}u^\ep +\theta (\partial_x u^\ep)^2 + \alpha\left(\frac{\partial_\theta u^\ep}{\ep}\right)^2 -M  \text{ on }\rr\times \Theta \times (0,\infty).
\end{equation}
\textbf{Step one:} Let us fix some $x_0\in \rr$ and $\theta_0\in \Theta$ with $n^0(x_0,\theta_0)>0$, and $R>0$ small so that $n^0(x,\theta)>0$ for all $(x,\theta)$ in the cube $K_R := (x_0-R, x_0+R)\times(\theta_0-R, \theta_0+R)$. Without loss of generality we may take $R<1/2$. Since $n_0$ is continuous and positive in $\bar{K}_R$, there exists a constant $\beta\leq 1$ such that $\inf_{\bar{K}_R} n_0 \geq \beta$. Therefore, if $(x,\theta)$ is contained in $\bar{K}_R$,  we have
\[
u^\ep(x,\theta,0)=\ep \ln n^0(x,\theta) \geq \ln (\beta).
\]
Let $C$ be the constant given by Lemma \ref{lem:barrier}, and define $a=M+(\theta_M +\alpha) C$. We define the function $\phi$ by,
\[
 \phi(x,\theta,t)= \xi(x-x_0)+\ep^2 \xi(\theta-\theta_0) -at+\ln(\beta)=\frac{1}{(x-x_0)^2-R^2} + \frac{\ep^2}{(\theta-\theta_0)^2-R^2} -at+\ln(\beta).
\]
Using Lemma \ref{lem:barrier}, it is easy to check that $\phi$ is a subsolution of (\ref{eqn:uepsuper}) in $K_R\times (0,\infty)$. Moreover, $\phi$ lies below $u^\ep$ on the parabolic boundary of $K_R \times (0,\infty)$. Indeed, for points $(x,\theta)\in \bar{K}_R$, we have $ \phi(x,\theta,0)\leq \ln(\beta)\leq u^\ep(x,\theta,0)$. And, if $(x,\theta)$ is contained in $\bdry K_R$, then $ \phi(x,\theta,t)=-\infty \leq u^\ep(x,\theta,t)$ for all $t>0$. Therefore, the comparison principle for equation (\ref{eqn:uepsuper}) implies 
\begin{equation}
\label{ugeqphi}
 u^\ep \geq \phi \text{ on }K_R\times (0,\infty).
 \end{equation} 
Let us fix some $T>0$. By (\ref{ugeqphi}) and the definition of $\phi$, we have,
\begin{equation}
\label{deftau}
 u^\ep(x,\theta,t) \geq -\frac{1}{4R^2} -\frac{\ep^2}{4R^2} -aT+\ln(\beta) =:\tau, \text{ for all }(x,\theta,t)\in \bar{K}_{R/2}\times [0,T].
\end{equation}
\textbf{Step two:} We will now use this lower bound to build a barrier for $u^\ep$ on $((\rr\times \Theta)\setminus K_{R/2}) \times (0,T)$. We define the function $\psi$ by,
\begin{equation*}
 \psi(x, \theta,t) = -\frac{b (x-x_0)^2}{t} - c t  +\tau.
\end{equation*}
We see that for $b$ and $c$ properly chosen and depending only on $T$, $R$, $\theta_m$, $\theta_M$, $\alpha$ and $r$, we have  that $\psi$ is a subsolution of (\ref{eqn:uepsuper}) on  $((\rr\times \Theta)\setminus K_{R/4}) \times (0,T)$. 
Moreover, we have $ \psi(x, \theta,t)\leq \tau$ for all $(x,\theta,t)$.  In particular, if $(x,\theta)\in \bdry K_{R/2}$ and $t\in [0,T]$, then, according to (\ref{deftau}), we have
\[
u^\ep(x,\theta,t) \geq \tau \geq \psi(x, \theta,t).
\]
Since $\psi_\theta \equiv 0$ and $\psi(x,\theta,0)=-\infty$ for all $(x,\theta)$, a comparison principle argument similar to that of  \cite[Lemma 2.1]{EvansSouganidis}  implies 
\[
 u^\ep \geq \psi \text{ on } ( (\rr\times \Theta)\setminus K_{R/2} )\times (0,T).
\]
Together with (\ref{ugeqphi}), the previous  estimate implies that if $Q\subset \subset \rr\times (0,\infty)$, then there exists a constant $C$ that depends on $Q$, $\alpha$, $r$, $\theta_m$ and $\theta_M$ such that for all  $0<\ep<1$ and for all $(x,t)\in Q $ and all $\theta\in\Theta$, we have
\[
 u^\ep(x,\theta,t)\geq -C.
\]
\end{proof}

\begin{proof}[Proof of the  upper bound of Proposition \ref{bd on utheta} ]
According to Corollary \ref{supbd:cor}, the estimate $\sup_{\rr\times\Theta\times(0,\infty)} n^\ep \leq C$ holds for a constant $C$ depending only on $\theta_m$, $\theta_M$, $\alpha$ and $r$. The upper bound on $u^\ep$ follows since, by definition we have $u^\ep = \ep \ln n^\ep$.
\end{proof}

\subsection{Proof of the gradient bound of Proposition \ref{bd on utheta}}
\label{subsec:grad}
\subsubsection{Idea of the proof -- strategy and challenges}
\label{subsubsect:ideagradient}
The proof of the gradient bound proceeds via a Bernstein argument -- in other words, we use that the derivatives of $u$ also satisfy certain PDEs in order to obtain estimates on them. We are most interested in obtaining an estimate on the derivative in $\theta$. To this end, let us denote $z=u^\ep_\theta$ and differentiate the equation  (\ref{eqn:uep}) for $u^\ep$ in $\theta$ to  find that $z$ satisfies,
\begin{equation}
\label{eq:demo}
z_t = \ep \theta z_{ x x} + \frac{\alpha}{\ep}z_{\theta \theta} + 2\theta u^\ep_x z_x +(u^\ep_x)^2 + \frac{2\alpha}{\ep^2} z z_\theta + \ep u^\ep_{xx}. 
\end{equation}
We bring the reader's attention to the last term, which involves a \emph{second} derivative of $u^\ep$ in the space variable $x$. This means that we cannot ``ignore" the space variable and must try to estimate $u^\ep_\theta$ and $u^\ep_x$ at the same time. In particular, the essense of our strategy is to consider the PDE satisfied not by $z$ but by $(u^\ep_\theta)^2 + \frac{1}{\ep^2}(u^\ep_x)^2$, and obtain estimates on this quantity.  

Our proof is similar to that of \cite[Lemma 2.4]{FS} and \cite[Lemma 2.1]{EvansSouganidis}, which use a Bernstein argument to obtain gradient bounds for  Hamilton-Jacobi equations with variable coefficients. However, our situation is more delicate because of the different scaling in the space and trait variable.  

\begin{rem}
\label{rem:gradbdBM}
The term $\ep u^\ep_{xx}$ in (\ref{eq:demo})  only arises because the diffusion coefficient is not constant in $\theta$. Let us compare our method to that of the proof of Lemma 2 of \cite{BouinMirrahimi}. We recall that the authors of \cite{BouinMirrahimi} consider the PDE (\ref{eqnBM}), which, after a rescaling and an exponential transformation, yields an equation similar to  our equation  (\ref{eqn:uep}) for $u^\ep$, but with constant diffusion coefficient. Thus, they are able to carry out the strategy we first describe -- essentially, they differentiate their PDE in $\theta$ and, because extraneous second derivative terms do not appear, they are able to use this equation to obtain the desired estimate. 
\end{rem}

\subsubsection{The proof of the gradient estimate}
Most of this subsection is devoted to the proof of:
\begin{prop}
\label{prop:gradbd}
Fix $0<\ep<1$. We denote $(x,\theta)$ by $y$. Suppose $u\in C^3(\rr\times\Theta\times (0,\infty))$ satisfies
\begin{equation}
\label{eqnu}
\begin{cases}
\partial_t u = \tr(A(y)D^2u)+G(Du,y)+f(y) \text{ on }\rr\times \Theta\times (0,\infty),\\
\partial_\theta u (x,\theta_m,t) =\partial_\theta u(x,\theta_M,t) = 0 \text{ for all }(x,t)\in \rr\times (0,\infty),
\end{cases}
\end{equation}
where the coefficients $A$ and $G$ are given by
 \[
A(y)=\left(\begin{matrix}
\ep y_2 &0\\
0&\frac{\alpha}{\ep}
\end{matrix}\right),  \  \   
G(p,y)=y_2 p_1^2+\frac{\alpha}{\ep^2}p_2^2
\]
and $f$ is independent of $y_2$ and satisfies
\begin{equation}
\label{linff}
-M\leq f\leq r
\end{equation} 
and
\begin{equation}
\label{fyasump}
\linfty{f_{y_1}}{\rr\times\Theta\times (0,\infty)}\leq \ep^{-1}M.
\end{equation} 
Given $Q\subset\subset Q' \subset\subset  \rr\times(0,T)$, there exists a constant $C$ that depends only on $Q$, $Q'$, $\linfty{u}{Q'\times \Theta}$, $\theta_m$, $\theta_M$, $\alpha$, $r$ and $M$ such that
\begin{equation}
\label{eq:lemuy2}
\sup_{Q\times\Theta} u_{y_2}^2\leq \ep C
\end{equation}
and
\begin{equation}
\label{eq:lemuy1}
\sup_{Q\times\Theta} u_{y_1}^2\leq \frac{C}{\ep}.
\end{equation}
\end{prop}

We remark that we only use the estimate (\ref{eq:lemuy2}) in the remainder of the paper. Nevertheless, our proof ``automatically" yields estimate (\ref{eq:lemuy1}) as well, so we state it for the sake of completeness. 

The gradient bound of Proposition \ref{bd on utheta} follows  from Proposition \ref{prop:gradbd} once we  verify its hypotheses:
\begin{proof}[Proof of the gradient bound of Proposition \ref{bd on utheta}]
We apply  Proposition \ref{prop:gradbd}  to $u^\ep$ with $f=r(1-\rho^\ep)$.  According to the supremum bound of Proposition \ref{bd on utheta}, which we have just established, we have that $u^\ep$ is bounded on $Q'\times \Theta$, uniformly in $\ep$. 

 According to Corollary \ref{supbd:cor} we have that $n^\ep$, and hence $\rho^\ep$, is uniformly bounded from above. In addition, $\rho^\ep$ is non-negative. Hence the hypothesis (\ref{linff}) is satisfied.
 
Now let us demonstrate that  (\ref{fyasump}) holds: since we have $f=r(1-\rho^\ep)$ and $y=(y_1,y_2)=(x,\theta)$, this amounts to establishing
\[
\linfty{\rho^\ep_x}{\rr\times\Theta\times (0,\infty)}\leq \ep^{-1}M
\]
for some constant $M$. Let us define $n(x,\theta,t)= n^\ep(\ep x,\theta,\ep t)$. Since $n^\ep$ satisfies (\ref{nep}) and $n_0$ satisfies (A\ref{asm n0}), we have that $n$ satisfies (\ref{n}) with initial data that satisfies (A\ref{asm n0}). Thus, according to Proposition \ref{deriv bd in terms of M} and Theorem \ref{prop:supbd}, $n$ is bounded in $C^{2,\eta}$. Therefore, there exists a constant $C$ so that, for all $(x,\theta,t)$,
\[
|n^\ep_x(x,\theta,t)| = \frac{1}{\ep} \left|n_x\left(\frac{x}{\ep},\theta,\frac{t}{\ep}\right) \right| \leq \frac{C}{\ep}.
\]
Integrating in $\theta$ yields the desired estimate for $\rho^\ep_x$. 
\end{proof}

For the proof of Proposition \ref{prop:gradbd} we will need the following elementary facts. We summarize them as a lemma and omit the proof.
\begin{lem}
\label{matrixlem}
For any diagonal $n\times n$ positive definite matrix $M$, any $n\times n$ matrices $X$, $Y$, and any $\beta>0$, we have
\begin{equation}
\label{matrixineqlem}
|\tr(MXY)|\leq \beta \tr(MXX^T)+\frac{1}{4\beta }\tr(MYY^T).
\end{equation}
In addition, if $X$ is symmetric then $\tr(MXMX)= \tr(MMXX)$.
\end{lem}



\begin{proof}[Proof of Proposition \ref{prop:gradbd}]
Let  $0\leq\psi(y_1,t)\leq 1$ be a cutoff function supported on $Q'$ and identically $1$ inside $Q$. Let us define $\zeta(y,t) = \psi^2(y_1,t)$ and 
\[
\tG(p,y,t)=\zeta(y,t)G(p,y).
\]
 We record for future use the following facts:
\begin{equation}
\label{HpDu}
G_p(Du, y)\cdot Du=2G(Du,y)
\end{equation}
\begin{equation}
\label{tHpp}
\tilde{G}_{pp}(Du,y)
= \frac{2\zeta}{\ep}A(y), \text{ and }\,  \,   
 \tilde{G}_{py}(Du,y)
=\frac{2\psi}{\ep}A(y)\cdot Du\otimes D\psi +2 \zeta W,
\end{equation}
where we denote $W=\left(\begin{matrix}
0 &u_{y_1}\\
0&0
\end{matrix}\right)$. Throughout the remainder of the argument we will use $C$ to denote constants that may depend on $\alpha$, $r$, $\linfty{u}{Q'}$, $\theta_m$, $\theta_M$, $Q$ and $M$ and that may change from line to line. We also define the constant $\lambda$ by
\[
\lambda=16\max\left\{ \linfty{\zeta_t}{Q'}, \frac{8}{\theta_m}, 16\theta_M ||\psi_{y_1}||_\infty, \frac{8\alpha||\psi_{y_1}||_\infty}{\ep}, \frac{||D^2\zeta||_{\infty}\alpha}{\ep}, 2\theta_M\sup_{Q'}(\zeta u_{y_1}), \frac{2M}{\ep}, \frac{2M\theta_m}{\ep}, \frac{1}{2\alpha\ep} \right\}.
\]
We remark that $\lambda$ is of the form,
\[
\lambda =  \left\{\frac{C}{\ep}, C \sup_{Q'}(\zeta u_{y_1})\right\}.
\]
We define the function
\begin{equation}
\label{defz}
z(y,t)=\tilde{G}(Du,y,t)+\lambda u(y,t)
\end{equation}
and we claim 
\begin{equation}
\label{locmaxz}
\text{if $(y_0,t_0)$ is the maximum of $z$ on $Q'\times \Theta$ and $(y_0,t_0)\notin (\partial Q')\times \Theta$, then } 
G(Du(y_0,t_0),y_0)\leq C.
\end{equation}
\textbf{Why (\ref{locmaxz}) implies the  estimates (\ref{eq:lemuy2}) and (\ref{eq:lemuy1}).} Let us assume (\ref{locmaxz}) holds, and we will establish (\ref{eq:lemuy2}). 

\textbf{First step:} We will first prove the following upper bound on $\tG$ holds:
\begin{equation}
\label{tHatmax}
\tG(Du,y,t)\leq C+\lambda C \text{ for all }(y,t)\in Q'\times \Theta.
\end{equation}
To this end, we use   (\ref{defz}) and the fact that $(y_0,t_0)$ is the maximum of $z$ to obtain,
\begin{equation}
\label{tHzy}
 \tG(Du,y,t)= z(y,t)-\lambda u(y,t) \leq z(y_0,t_0)-\lambda u(y,t) \text{ for all }(y,t)\in Q'\times \Theta.
\end{equation}
Next we consider two cases:  the first is $(y_0,t_0)\in (\partial Q')\times \Theta$, and the second is  $(y_0,t_0)\notin (\partial Q')\times \Theta$.  In the first case, we have $\zeta(y_0,t_0)=0$, so we find $z(y_0,t_0)=\lambda u(y_0,t_0)$. Therefore, the estimate (\ref{tHzy}) says,
\[
\tG(Du,y,t) \leq \lambda u(y_0,t_0)-\lambda u(y,t) \text{ for all }(y,t)\in Q'\times \Theta.
\] 
Since we have $\linfty{u}{Q'}\leq C$, this establishes the estimate (\ref{tHatmax}) in the case  $(y_0,t_0)\in (\partial Q')\times \Theta$. 

Now let us suppose  $(y_0,t_0)\notin (\partial Q')\times \Theta$, 
so that, according to (\ref{locmaxz}), we have 
\begin{equation}
\label{Hy0t0}
G(Du(y_0,t_0),y_0)\leq C.
\end{equation}
We now  use  (\ref{defz}) to express the right-hand side of (\ref{tHzy}) in terms of $\tG$ and $u$:
\[
\tG(Du,y,t) \leq  \tG(Du,y_0,t_0) + \lambda u(y_0,t_0)-\lambda u(y,t).
\]
We use that $\tG = \zeta G \leq G$ and the estimate (\ref{Hy0t0}) to bound the first term on the right-hand side of the previous line of the above by $C$,  and that $\linfty{u}{Q'\times\Theta}\leq C$ to bound the last two terms. This yields  (\ref{tHatmax}).

\textbf{Second step:} Next we will use (\ref{tHatmax}) to verify the following upper bound on the derivatives of $u$ in terms of $C$ and $\lambda$:
\begin{equation}
\label{eq:lemalmost}
\theta_mu_1^2(y,t)+\frac{\alpha}{\ep^2} u_{y_2}^2(y,t)\leq C+\lambda C \text{ for all } (y,t)\in Q\times \Theta.
\end{equation}
To this end, we use  the definition of $\tG$ and the fact that $\zeta$ is identically $1$ inside $Q$ to obtain,
\[
\theta_mu_1^2(y,t)+\frac{\alpha}{\ep^2} u_{y_2}^2(y,t) \leq y_2u_1^2(y,t)+\frac{\alpha}{\ep^2} u_{y_2}^2(y,t) = \tG(Du, y, t) \text{ for all }(y,t)\in Q\times \Theta.
\]
We use  the estimate (\ref{tHatmax}) to bound the first term on the right-hand side of the previous line from above by $C+\lambda C$ and obtain the estimate (\ref{eq:lemalmost}).

\textbf{Third step:} There are two possible values for $\lambda$, and we will show that, together with the estimate (\ref{eq:lemalmost}) that we just established, either one implies the desired estimate (\ref{eq:lemuy2}). Let us first suppose $\lambda=C\ep^{-1}$. Upon substituting this on the right-hand side of (\ref{eq:lemalmost}) we find,
\[
\theta_mu_1^2(y,t)+\frac{\alpha}{\ep^2} u_{y_2}^2(y,t) \leq C+\frac{C}{\ep},
\]
which yields the desired estimates (\ref{eq:lemuy2}) and (\ref{eq:lemuy1}) by multiplying both sides by $\ep^2$.

Now let us suppose that $\lambda$ takes on the other possible value, so that $\lambda = C \left( \sup_{Q'}(\zeta u_{y_1})\right)$. We use that $\theta_m\leq y_2$, the definition of $\tG$, the line (\ref{tHatmax}) and the value of $\lambda$ to obtain, for any $(y,t)\in Q'\times \Theta$,
\[
\theta_m\zeta(y,t) u_{y_1}^2(y,t)\leq \zeta(y,t) y_2 u_{y_1}^2(y,t)\leq \tG(Du,y,t)\leq C+\frac{C}{\lambda} = C+ C\left( \sup_{Q'}(\zeta u_{y_1})\right) .
\]
From this we deduce $\sup_{Q'}(|\zeta u_{y_1}|)\leq C$ and hence we have $\lambda\leq C$ (and, since $\zeta \equiv 1$ on $Q$, (\ref{eq:lemuy1}) holds). Using this on the right-hand side of (\ref{eq:lemalmost}) implies that (\ref{eq:lemuy2}) holds.

\textbf{The proof of (\ref{locmaxz}).}
Let us suppose $(y_0,t_0)$ is the maximum of $z$ on $Q'\times \Theta$ and $(y_0,t_0)\notin (\partial Q')\times \Theta$. There are two cases to consider: the first is that $(y_0,t_0)$ is in the  interior of $Q'\times\Theta$,  and the second is that $(y_0,t_0)\in Q'\times \partial \Theta$. Let us tackle the first case, so that we have,
\begin{equation}
\label{atmaxz}
0\leq z_t(y_0,t_0)-\tr (AD^2_{yy}z(y_0,t_0)), \text{ and } D_y z(y_0,t_0)=0.
\end{equation}

We seek to establish
\begin{equation}
\label{wantH}
G(Du(y_0,t_0),y_0)\leq C.
\end{equation}
As we have just shown, once we establish (\ref{wantH}) the proof of the proposition will be complete. To this end, we compute,
\begin{equation}
\label{eqforz}
z_t-\tr(AD^2_{yy} z) = \lambda (u_t - \tr(AD^2u))+\zeta G_p(Du, y)\cdot (Du_t -\tr(AD^3u))+\text{I}
\end{equation}
where $\text{I}$ is the sum of the left-over terms from $z_t$ and $-\tr(AD^2z)$:
\[
\text{I} = \tG_t(Du,y,t) - \tr\left[A\left(\tG_{pp}(Du,y,t) D^2u D^2u +2\tG_{py}(Du,y,t) D^2u +\tG_{yy}(Du,y,t)\right)\right].
\]
(Throughout, we use $D$ to denote the derivative in $y$.) 
We  use that $u$ satisfies (\ref{eqnu}) to write the first term on the right-hand side of (\ref{eqforz}) as
\begin{equation}
\label{firstterm}
\lambda (u_t - \tr(AD^2u)) = \lambda (G(Du,y)+f(y))\leq \lambda (G(Du,y)+r),
\end{equation}
where we have used (\ref{linff}) to obtain the inequality. 
Now let us look at the second term on the right-hand side of (\ref{eqforz}). We recognize that $Du_t -\tr(AD^3u)$ is ``almost" the derivative of $u_t -\tr(AD^2u)$, up to a term that involves a derivative of $A$. In addition, rearranging the equation that $u$ solves implies $u_t -\tr(AD^2u)= G(Du,y)+f(y)$. We find:
\[
Du_t -\tr(AD^3u) =D(u_t -\tr(AD^2u)) +  \tr(DAD^2u) = D(G(Du,y)) +Df(y)+\tr(DAD^2u) .
\]
Multiplying by $\zeta$ we obtain,
\begin{equation}
\label{secondterm}
\zeta (Du_t -\tr(AD^3u)) = \zeta D(G(Du,y)) +\zeta Df(y)+\zeta  \tr(DAD^2u) .
\end{equation}
So far, our computations hold on all of $Q'\times \Theta$. Now we will specialize to $(y_0,t_0)$ and obtain an alternate expression for the first term on the right-hand side of the previous line. We  recall the derivative of $z$ is zero at $(y_0,t_0)$, so that, 
\begin{equation}
\label{Dz0}
0=Dz(y_0,t_0)= D(\tG(Du,y_0))+\lambda Du= G(Du, y_0)D\zeta  +\zeta D(G(Du,y_0)) +\lambda Du,
\end{equation}
which, upon rearranging becomes,
\[
\zeta D(G(Du,y_0)) =- \lambda Du-  G(Du, y_0)D\zeta .
\]
We  substitute the right-hand side of the previous line for the first term of the right-hand side of (\ref{secondterm}) to obtain, at $(y_0,t_0)$,
\begin{equation*}
\zeta (Du_t -\tr(AD^3u) )=- \lambda Du- G(Du, y_0)D\zeta + \zeta Df(y_0)+\zeta \tr(DAD^2u).
\end{equation*}
We take dot product with $G_p(y_0,t_0)$ and use the previous line to find: 
\begin{equation}
\label{secondtermuse}
\zeta G_p(Du, y_0)\cdot (Du_t -\tr(AD^3u)) = -\lambda G_p(Du, y_0)\cdot Du + \text{II}
\end{equation}
where $\text{II}$ is the sum of the left-over terms:
\[
\text{II} = G_p(Du,y_0)\cdot (- G(Du, y_0) D\zeta  + \zeta Df(y_0)+ \zeta \tr(DAD^2u)).
\]
Next, according to (\ref{HpDu}), we have  $G_p(Du, y)\cdot Du= 2G(Du, y)$. We use this on the right-hand side of (\ref{secondtermuse}) and find,
\begin{equation}
\label{secondtermagain}
\zeta G_p(Du, y_0)\cdot (Du_t -\tr(AD^3u)) = -2\lambda G(Du,y_0) + \text{II}.
\end{equation}
Let us now consider (\ref{eqforz}) evaluated at $(y_0,t_0)$. According to (\ref{atmaxz}),  the left-hand side of  (\ref{eqforz}) is non-negative. We   use (\ref{firstterm}) to estimate the first term on the right-hand side of (\ref{eqforz}), and we use (\ref{secondtermagain}) for the second term, and find,
\begin{equation}
\label{HandDpH}
0\leq 
\lambda (G(Du,y_0)+r- 2G(Du,y_0))+ \text{II}+\text{I} = -\lambda G(Du,y_0) +r\lambda + \text{II} +\text{I}.
\end{equation}
We now claim that the sum of the leftover terms $\text{I}$ and $\text{II}$ is bounded:
\begin{equation}
\label{esterror}
\text{II}+\text{I}\leq C\lambda+ \frac{\lambda}{2} G(Du,y) .
\end{equation}
We point out that once the bound (\ref{esterror}) is established, the proof of (\ref{wantH}), and hence of the proposition, will be complete. Indeed, using (\ref{esterror}) to estimate the right-hand side of (\ref{HandDpH}) yields 
\[
0\leq - \frac{\lambda}{2} G(Du(y_0,t_0), y_0)+C\lambda,
\]
which, upon rearranging and dividing by $\lambda>0$ yields (\ref{wantH}).

\textbf{Proof of bound (\ref{esterror}) on error terms.} 
 Let us start with $\text{I}$. We will prove,
\begin{equation}
\label{estonstar}
 \text{I}\leq - \frac{\zeta}{\ep}\tr\left[A^2 \cdot D^2u\cdot D^2u \right] +\frac{\lambda}{4} G(Du,y).
\end{equation}
We use the  expressions (\ref{tHpp}) for $\tG_{pp}$ and $\tG_{py}$ to  rewrite $\text{I}$ as
\begin{equation}
\label{traceinstar}
\begin{split}
\text{I} =\tG_t(Du,y) - \frac{2\zeta}{\ep} \tr\left[A^2 D^2uD^2u\right] - \frac{4\psi}{\ep}\tr\left[A^2\cdot Du\otimes D\psi  D^2u \right] 
-4\zeta\tr\left[ A W D^2u \right] -\tr\left[A\tG_{yy}\right].
\end{split}
\end{equation}
Let us bound from above the third term in $\text{I}$ by applying the  inequality of Lemma \ref{matrixlem} with $M=A^2$, $X=Du\otimes D\psi$, $Y=\psi D^2u$ and $\beta=2$. We obtain: 
\begin{equation}
\label{estwithD2u}
- \frac{4\psi}{\ep}\tr\left[A^2 Du\otimes D\psi  D^2u \right]\leq \frac{8}{\ep} \tr(A^2 (Du\otimes D\psi)(Du\otimes D\psi)^T)+\frac{\psi^2}{2\ep}\tr(A^2 D^2u D^2u).
\end{equation}
Let us apply Lemma \ref{matrixlem}  again, this time with $M=Id$, $X=W$, $Y=AD^2u$ and $\beta=2$. We obtain the following upper bound for the fourth term in $\text{I}$:
\[
- \zeta \tr\left[A W D^2u \right]\leq 8 \zeta  \tr\left(WW^T\right)+\frac{\zeta}{2}\tr(A^2 D^2u D^2u) = 8\zeta u_1^2  +\frac{\zeta}{2}\tr(A^2 D^2u D^2u),
\]
where we have also used the second statement in Lemma \ref{matrixlem}. We use  (\ref{estwithD2u}) and the previous line to bound from above the third and fourth terms, respectively, in on the right-hand side of (\ref{traceinstar}). Notice that the terms involving $A^2 D^2u D^2u$ in the previous line and in (\ref{estwithD2u}) will be ``absorbed" by the second term in $\text{I}$ (we are using that $\psi^2=\zeta$). We obtain:
\begin{equation}
\label{bdforRHStracealmost}
\text{I}\leq \tG_t(Du,y) - \frac{\zeta}{\ep}\tr\left[A^2  D^2u D^2u \right] +\frac{8}{\ep}\tr(A^2 (Du\otimes D\psi)(Du\otimes D\psi)^T) + 8\zeta u_{y_1}^2 -\tr\left[A\tG_{yy}(Du,y)\right].
\end{equation}
The first term on the right-hand side of (\ref{bdforRHStracealmost}) is simply $\zeta_t G(Du,y)$, which is less than $\frac{\lambda}{16} G(Du,y)$. In addition, we have $8\zeta u_{y_1}^2\leq 8 \frac{\theta}{\theta_m}u_{y_1}^2\leq \frac{\lambda}{16}G(Du,y)$. We use this to bound the right-hand side of (\ref{bdforRHStracealmost}) from above and find,
\begin{equation}
\label{bdforRHStrace}
\text{I}\leq  \frac{\lambda}{8}G(Du,y)- \frac{\zeta}{\ep}\tr\left[A^2  D^2u D^2u \right] +\frac{8}{\ep}\tr(A^2 (Du\otimes D\psi)^2) -\tr\left[A\tG_{yy}(Du,y)\right].
\end{equation}
We will show now show that each of the last two terms on the right-hand side of the previous line is less than $\frac{\lambda}{16} G(Du, y)$. Once we show this, the estimate (\ref{estonstar}) will be established. To this end, we use that $\psi$ is independent of $y_2$, to compute
\[
\tr(A^2 (Du\otimes D\psi)(Du\otimes D\psi)^T) = y_2^2 \ep^2 u^2_{y_1}\psi^2_{y_1} +\frac{\alpha^2}{\ep^2} u_{y_2}\psi_{y_1}^2. 
\]
Multiplying by $8/\ep$ and using the definitions of $G$ and $\lambda$ we obtain,
\begin{equation*}
\frac{8}{\ep}\tr(A^2 (Du\otimes D\psi)(Du\otimes D\psi)^T)
\leq 8\theta_M ||\psi_{y_1}||_\infty y_2 u_{y_1}^2+ \frac{8\alpha||\psi_{y_1}||_\infty}{\ep} \frac{\alpha u_{y_2}^2 }{\ep^2}
\leq \frac{\lambda}{16}G(Du, y).
\end{equation*}
Let us estimate  the last term of (\ref{bdforRHStrace}). We compute $\tG_{yy}$ in terms of the derivatives of $G$ and $\zeta$ and find, 
\[
\tG_{yy}=G_y\otimes D\zeta +D\zeta \otimes G_y+ G D^2\zeta = 
\left(\begin{matrix}
0&u_{y_1}^2\zeta_{y_1}\\
u_{y_1}^2\zeta_{y_1}&0
\end{matrix}\right) 
+ G \left(\begin{matrix}
0&0\\
0&\zeta_{y_2 y_2}
\end{matrix}\right)
\]
We multiply both sides of the previous line by $-A(y)$ and take trace. The first term gives zero. The second is simply $-G(Du,y)\zeta_{y_2 y_2} \frac{\alpha}{\ep}$. We summarize this as,
\begin{align*}
-\tr(A\tilde{G}_{yy}(Du,y))=-G(Du,y)\zeta_{y_2 y_2} \frac{\alpha}{\ep} \leq \frac{||D^2\zeta||_\infty \alpha}{\ep} G(Du,y)
& \leq \frac{\lambda}{16}G(Du, y).
\end{align*}
Thus we have proved (\ref{estonstar}). Now for $\text{II}$. We have
 \[
\text{II} = -G_p(Du,y)\cdot  D\zeta  G(Du, y)  + \zeta G_p(Du,y)\cdot Df(y)+ \zeta G_p(Du,y)\cdot  \tr(DAD^2u)).
\]
Since  since $D\zeta=(\zeta_{y_1},0)$, the first term in $\text{II}$ is simply
\[
- 2y_2 \zeta u_{y_1} G(Du, y) \leq 2\theta_M \left( \sup_{Q'}(\zeta u_{y_1})\right)G(Du, y)\leq \frac{\lambda}{16}G(Du,y).
\]
For the second term in $\text{II}$, we use that $Df=(f_{y_1},0)$,  the Cauchy-Schwarz inequality, the assumption (\ref{fyasump}), and the definitions of $G$ and $\lambda$ to find,
\[
\zeta G_p(Du,y)\cdot Df(y) = \zeta 2y_2  u_{y_1} f_{y_1} \leq 2y_2 \zeta (u_{y_1}^2+1) \linfty{f_{y_1}}{Q'}
\leq \frac{2M}{\ep}(G(Du,y)+\theta_M)  \leq \frac{\lambda}{16}G(Du,y)+ \frac{\lambda}{16},
\]
Finally we will bound the last term in $\text{II}$. We have $D_{y_1}A \equiv 0$, so the last term is simply
\begin{align*}
\zeta G_{p_2}(Du, y) \tr(D_{y_2}A D^2u) &= \zeta\frac{ u_{y_2}}{\ep^2}  y_2 \ep u_{y_1,y_1} 
\leq \frac{\zeta}{\ep}\left( \frac{u_{y_2}^2}{2\ep}+  \frac{\ep^2y_2^2u_{y_1 y_1}^2}{2}\right)\\
&\leq \frac{\zeta}{\ep}\left(\frac{1}{2\alpha} G(Du,y) +\frac{1}{2}\tr(A^2D^2uD^2u)\right)\\
&\leq \frac{\lambda}{16} G(D^2u,y) + \frac{\zeta}{2\ep} \tr(A^2D^2uD^2u)
\end{align*}
where the first inequality follows by applying Cauchy Schwarz, the second from the definitions of $A$ and $G$, and the third from our choice of $\lambda$. We therefore find,
\[
\text{II}\leq \frac{3\lambda}{16}G(Du,y) +\frac{\lambda}{16}+ \frac{\zeta}{2\ep} \tr(A^2D^2uD^2u).
\]
Adding our upper bound (\ref{estonstar}) on $\text{I}$ to the previous line we obtain,
\[
\text{I}+\text{II}\leq \frac{\lambda}{2}G(Du,y) +\frac{\lambda}{16} -  \frac{\zeta}{2\ep} \tr(A^2D^2uD^2u) \leq  \frac{\lambda}{2}G(Du,y) +\frac{\lambda}{16} ,
\]
as desired. Thus our analysis of the case that $(y_0,t_0)$ is in the  interior of $Q'\times\Theta$ is complete. 

\textit{Completing the proof of (\ref{locmaxz}): analysis of the second case.} We now tackle the second case, and suppose  $(y_0,t_0)\in Q'\times \partial \Theta$. We do not have that (\ref{atmaxz}) holds; let us first determine the ``replacement" for  (\ref{atmaxz}) in this case. Since the $y_1$ and $t$ coordinates of the maximum $(y_0,t_0)$ of $z$ are interior, we have $z_t=0$, $z_{y_1}=0$, and $z_{y_1 y_1}\leq 0$ at $(y_0,t_0)$.  We now need to consider $z_{y_2}$ and $z_{y_2 y_2}$. To this end, we use the definitions of $z$ and $G$ to compute,
\[
z_{y_2}=\zeta (u_1^2+2y_2u_1u_{12}+2\frac{\alpha}{\ep}u_{22}u_2)+ \lambda u_2.
\]
Since $u_2\equiv 0$ in $\bdry \Theta$, we have $u_{12}\equiv 0$ on $\bdry \Theta$ as well.  Hence,  on $Q'\times \partial \Theta$, we find
\[
z_{y_2}(y,t) = \zeta u_1^2.
\]
Taking another derivative of $z$ in  $y_2$ and using that $u_{12}\equiv 0$ on $\bdry \Theta$, we obtain
\[
z_{y_2 y_2}\equiv 0 \text{ on }\bdry \Theta.
\]
Putting everything together, we find that instead of  (\ref{atmaxz}) we have,
\begin{equation}
\label{atmaxzsecond}
0\leq z_t(y_0,t_0)-\tr (AD^2_{yy}z(y_0,t_0)), \text{ and } D z(y_0,t_0)=(0,\zeta u_1^2).
\end{equation}
We recall that we used (\ref{atmaxz}) in two places. The first was to find that the left-hand side of  (\ref{eqforz}) is non-negative at $(y_0,t_0)$, which allowed us to obtain the lower bound of $0$ for the left-hand side of (\ref{HandDpH}). We see that, since (\ref{atmaxzsecond}) holds, we still have this lower bound. The second place where we used (\ref{atmaxz}) was line (\ref{Dz0}), to say  $Dz(y_0,t_0)=0$. Now, we instead have $D z(y_0,t_0)=(0,\zeta u_1^2)$.  Let us proceed as in the proof of the first case (we will see that, because we will soon take dot product with a vector of the form $(\cdot, 0)$, this extra term will not affect the argument). Line (\ref{Dz0}) no longer has $0$ on the left-hand side but becomes,
\[
(0,\zeta u_1^2)=Dz(y_0,t_0)= D(\tG(Du,y_0))+\lambda Du= G(Du, y_0)D\zeta  +\zeta D(G(Du,y_0)) +\lambda Du.
\]
We  rearrange and find,
\[
\zeta D(G(Du,y_0)) =- \lambda Du-  G(Du, y_0)D\zeta +(0,\zeta u_1^2).
\]
The next step is to substitute the right-hand side of the previous line for the first term of the right-hand side of (\ref{secondterm}). Doing this we obtain, at $(y_0,t_0)$,
\begin{equation*}
\zeta (Du_t -\tr(AD^3u) )=- \lambda Du- G(Du, y_0)D\zeta +(0,\zeta u_1^2) + \zeta Df(y_0)+\zeta \tr(DAD^2u).
\end{equation*}
We take dot product with $G_p(y_0,t_0)$. However, the key here is that $G_p(y_0,t_0)$ is simpler than before: we use the definition of $G$ to compute,
\[
G_p(y_0,t_0)= (2y_2u_1,2\frac{\alpha}{\ep^2}u_2)=(2y_2u_1, 0),
\]
where the second inequality follows since $(y_0, t_0)\in Q'\times \bdry\Theta$, and $u_2\equiv 0$ on $\bdry \Theta$. Hence, the term $(0,\zeta u_1^2)$ disappears upon taking dot product with $G_p(y_0,t_0)$, and therefore (\ref{secondtermuse}) remains unchanged. Thus the remainder of the argument goes through as in the first case, and the proof of Proposition \ref{prop:gradbd} is complete.

\end{proof}

\section{Limits of the $u^\ep$ as $\ep$ approaches zero}
\label{sec:limitsuep}
This section is devoted to studying the half-relaxed limits $\bar{u}$ and $\underline{u}$ of the $u^\ep$ that we mentioned in the introduction. For the convenience of the reader, we recall their definitions here:
\[
\bar{u}(x,t)=\lim_{\ep\rightarrow 0} \sup \{u^{\ep'}(y,\theta,s):\,  \ep'\leq \ep,  |y-x|,  |t-s|\leq \ep, \theta\in \Theta \}
\]
and
\[
\underline{u}(x,t)=\lim_{\ep\rightarrow 0}\inf\{u^{\ep'}(y,\theta,s):\, \ep'\leq \ep, |y-x|,  |t-s|\leq \ep, \theta\in \Theta\}.
\]
We also summarize the results of the previous sections. We have established Corollary \ref{supbd:cor}, which says that there exists a constant $C$ so that $\linfty{n^\ep}{\rr\times\Theta\times [0,\infty)}\leq C$. We have also proved Proposition \ref{bd on utheta}, which says that if $Q\subset \subset \rr\times  (0,\infty)$, then there exists a constant $C$ that depends on $Q$ such that for  $(x,t)\in Q$ and $\theta\in \Theta$, we have
\[
 u^\ep(x,\theta,t)\geq -C
\]
and
\[
|u^\ep_\theta (x,\theta, t)|\leq \ep^{1/2} C.
\]
In particular, these results imply that $\bar{u}$ and $\underline{u}$ are finite everywhere on $\rr\times (0,\infty)$ (although they may be infinite at time $t=0$, as we will demonstrate in the proof of Proposition \ref{main result}). In addition, we obtain
that $\bar{u}$ and $\underline{u}$ are non-positive: indeed, since $\linfty{n^\ep}{\rr\times\Theta\times [0,\infty)}\leq C$, we have $u^\ep(x,\theta,t) \leq \ep \ln C$ for any $(x,\theta, t)\in \rr\times [0,\infty)\times \Theta$. Taking $\liminf$ or $\limsup$ as $\ep\rightarrow 0$ implies,
\begin{equation}
\label{sign}
\bar{u}(x,t)\leq 0 \text{ and } \underline{u}(x,t)\leq 0 \text{ for all }(x,t)\in \rr\times [0,\infty).
\end{equation}

For the reader's convenience, we recall what it means for $u$ to be a sub- or super- viscosity solution of the constrained Hamilton-Jacobi equation (\ref{HJ}). 
\begin{defn}
We say $u$ is a sub- (resp. super-) viscosity solution of 
\[
\max\{u_t - H(Du), u\} =0 \text{ on }\rr\times (0,\infty)
\]
if, for any point $(x,t)$ and any smooth test function $\phi$ such that $u-\phi$ has a local maximum (resp. minimum) at $(x,t)$, we have 
\[
u(x,t) \leq 0 \text{ and } \phi_t(x,t)-H(D\phi) \leq 0
\]
\[
(\text{resp. }  u(x,t) \geq 0 \text{ or } \phi_t(x,t)-H(D\phi) \geq 0).
\]
\end{defn}

In subsection \ref{sec:pf main prop} we prove that $\bar{u}$ and $\underline{u}$ are, respectively, a sub- and super- viscosity solution of the Hamilton-Jacobi equation (\ref{HJ}). This is part Part (I) of Proposition \ref{main result}. We use  a perturbed test function  argument \cite{EvansPTF} and some techniques similar to the proofs of \cite[Theorem 1.1]{EvansSouganidis}, \cite[Propositions 3.1 and 3.2]{BarlesEvansSouganidis}, and \cite[Proposition 1]{BouinMirrahimi}. We will perturb our test functions by $\ep \ln Q(\theta)$, where  $Q$ satisfies the spectral  problem (\ref{spectral}) for an appropriate $\lambda$. From first glance this is slightly different from the perturbations in, say, \cite{EvansPTF}, due to the presence of  $\ln$; however, this is natural seeing as we have defined $u^\ep=\ep \ln n^\ep$.

In subsection \ref{sec:at time 0} we study the behavior of $\bar{u}$ and $\underline{u}$ at $t=0$ and establish Part (II) of Proposition \ref{main result}. We  follow the strategy of \cite[Propositions 3.1, 3.2]{BarlesEvansSouganidis}.

\subsection{Proof of Part (I) of Proposition \ref{main result}}
\label{sec:pf main prop}
Throughout this section we employ the notational convention we mentioned in Section \ref{notation}: 
if $Q\subset\rr\times (0,\infty)$, then we will use  $Q\times \Theta$ to denote,
\[
Q\times \Theta = \{(x,\theta,t):\  \ (x,t)\in Q \text{ and } \theta\in \Theta\}.
\]

\subsubsection{An auxillary lemma}
We formulate the following lemma, which is similar  to \cite[Lemma 6.1]{User's guide}.
\begin{lem}
\label{lem:u*}
Let $v(x,t)$ and $Q(\theta)$ be smooth functions. Suppose $\underline{u}-v$  (resp. $\bar{u}-v$) has a local minimum (resp. maximum) at $(x_0,t_0)$. Define $v^\ep(x,\theta,t)= v(x,t) + \ep Q(\theta)$. Then there exists a subsequence $\{\ep_n\}_{n=1}^{\infty}$ such that 
\begin{enumerate}
\item \label{itemlimit} $(x_{\ep_n},t_{\ep_n}) \rightarrow (x_0,t_0)$,
\item \label{itemmin} $u^{\ep_n} - v^{\ep_n}$ has a local minimum (resp. maximum) at $(x_{\ep_n},\theta_{\ep_n},t_{\ep_n})$ for some $\theta_{\ep_n}$, and
\item \label{itemu} $\lim_{n\rightarrow \infty} u^{\ep_n} (x_{\ep_n},\theta_{\ep_n},t_{\ep_n}) = \underline{u}(x_0,t_0)$ (resp. $=\bar{u}(x_0,t_0)$).
\end{enumerate}
\end{lem}
We postpone the proof of the lemma until the end of  subsection \ref{sec:pf main prop}. 
\subsubsection{Proof that $\underline{u}$ is a supersolution of (\ref{HJ})} We place one computation into a separate lemma.
\begin{lem} 
\label{lem:for main prop}
Suppose $v(x,t)$ is a smooth function that satisfies, for some $(x_0,t_0)$,
\begin{equation}
\label{v at x0t0}
v_t(x_0,t_0)-H(\partial_x v(x_0,t_0)) = -a .
\end{equation}
We set $\lambda =\partial_x v(x_0,t_0)$ and take $Q(\theta)$  to be the solution of the spectral problem (\ref{spectral}) corresponding to this $\lambda$. We define
\begin{equation}
\label{def vep}
v^\ep(x, \theta,t)= v(x,t) + \ep \ln Q(\theta).
\end{equation}
There exist positive constants $s$, $\ep_1$ such that for all $(x, \theta)\in B_s(x_0,t_0)$, all $\theta\in\Theta$ and all $0<\ep\leq \ep_1$ we have 
\begin{align}
\label{eqn for vep}
v^\ep_t (x,\theta,t)&- \ep \theta v^\ep_{xx} (x,\theta,t)- \theta (v^\ep_x)^2 (x,\theta,t)-\frac{\alpha}{\ep}v^\ep_{\theta \theta}(x,\theta,t )  - \frac{\alpha}{\ep^2} (v^{\ep}_\theta(x,\theta,t))^2 - r \leq -a/2.
\end{align}
\end{lem}
We postpone the proof of  the lemma  and proceed with:
\begin{proof}[Proof that $\underline{u}$ is a supersolution of (\ref{HJ})]
Let us fix a point $(x_0,t_0)\in \rr\times (0,\infty)$ such that $\underline{u}(x_0,t_0)<0$.  By Proposition  \ref{bd on utheta}, $\underline{u}(x_0,t_0)$ is finite, so we denote $-\delta =\underline{u}(x_0,t_0)$. We aim to prove that $\underline{u}$ is a supersolution of (\ref{HJ}) in the viscosity sense, so let us suppose $\underline{u}-v$ has a local minimum at $(x_0,t_0)$ for some smooth function $v$. We shall show
\begin{equation}
\label{supersol}
v_t(x_0,t_0)-H(\partial_x v(x_0,t_0)) \geq 0. 
\end{equation}
We proceed by contradiction and assume that (\ref{supersol}) does not hold. Therefore, there exists $a>0$ such  that $v$ satisfies (\ref{v at x0t0}). 
We define the perturbed test function $v^\ep(x,\theta,t)$ by (\ref{def vep}). According to Lemma \ref{lem:for main prop}, there exist $s>0$ and $\ep_1>0$ such that $v^\ep$ satisfies (\ref{eqn for vep}) on $B_s(x_0,t_0)\times \Theta$ for all $\ep\leq \ep_1$.

By Proposition \ref{bd on utheta}, there exists some positive constants $c$, $C$, $\ep_0$ such that for all $\ep<\ep_0$ and for all $(x,t,\theta)\in B_{c/2}(x_0,t_0)\times \Theta$,
\begin{equation}
\label{utheta specific}
|u^\ep_\theta (x,\theta, t)|\leq \ep^{1/2} C.
\end{equation}
Let $(x_{\ep_n},\theta_{\ep_n},t_{\ep_n})$ be the sequence of points given by Lemma \ref{lem:u*}.  According to item (\ref{itemu}) of Lemma \ref{lem:u*}, we have
\[
\lim_{n\rightarrow \infty} u^{\ep_n} (x_{\ep_n},\theta_{\ep_n},t_{\ep_n}) = \underline{u}(x_0,t_0) = -\delta.
\]
Therefore, for all $n$ large enough,  we find,
 \[
 u^{\ep_n} (x_{\ep_n},\theta_{\ep_n},t_{\ep_n})\leq -\delta/2.
 \] 
Together with the estimate (\ref{utheta specific}) on $\partial_\theta u^\ep$, this implies that there exists $N_1>0$ such that for $n>N_1$ and for all $\theta\in \Theta$,
 \[
 u^{\ep_n} (x_{\ep_n},\theta,t_{\ep_n})\leq -\delta/4.
 \] 
 The previous estimate implies that $\rho^{\ep_n}(x_{\ep_n},t_{\ep_n}) $ is bounded from above, uniformly in $n$. Indeed, we use the definition of $\rho^{\ep_n}$, the relationship $n^{\ep_n}(x,\theta,t)= e^{\frac{u^{\ep_n}(x,\theta,t)}{\ep_n}}$ and the previous line to obtain
 \[
 \rho^{\ep_n}(x_{\ep_n},t_{\ep_n}) = \int e^{\frac{u^{\ep_n}(x_{\ep_n},\theta,t_{\ep_n})}{\ep_n}} \, d\theta \leq   \int e^{\frac{-\delta}{4\ep_n}} \, d\theta.
 \]
Thus there exists $N_2$ such that for all $n>N_2$, we have
\begin{equation}
\label{rho tiny}
 \rho^{\ep_n}(x_{\ep_n},t_{\ep_n}) \leq \frac{a}{4r}.
 \end{equation}
 
Since  $u^{\ep_n} - v^{\ep_n}$ has a local minimum at $(x_{\ep_n},\theta_{\ep_n},t_{\ep_n})$ 
\begin{equation}
\label{txderivs}
v^{\ep_n}_t(x_{\ep_n},\theta_{\ep_n},t_{\ep_n})=u^{\ep_n}_t(x_{\ep_n},\theta_{\ep_n},t_{\ep_n}), \  \  v^{\ep_n}_{xx}(x_{\ep_n},\theta_{\ep_n},t_{\ep_n})\leq u^{\ep_n}_{xx}(x_{\ep_n},\theta_{\ep_n},t_{\ep_n}),
\end{equation}
and
\begin{equation}
\label{thetaderivs}
v^{\ep_n}_{\theta} (x_{\ep_n},\theta_{\ep_n},t_{\ep_n})= u^{\ep_n}_{\theta}(x_{\ep_n},\theta_{\ep_n},t_{\ep_n}), \  \   v^{\ep_n}_{\theta \theta} (x_{\ep_n},\theta_{\ep_n},t_{\ep_n})\leq u^{\ep_n}_{\theta \theta}(x_{\ep_n},\theta_{\ep_n},t_{\ep_n}).
\end{equation}
We remark that (\ref{thetaderivs}) holds both if $\theta_{\ep_n}$ is an interior point of $\Theta$ and if $\theta_{\ep_n}$ is a boundary point of $\Theta$. Indeed, let us suppose that $\theta_{\ep_n}\in \bdry \Theta$. By the definition of $v^\ep$ in  (\ref{def vep}), we have $v^\ep_\theta = \ep\frac{Q_\theta}{Q}$, where $Q$ satisfies (\ref{spectral}). Since $Q$ is positive and satisfies Neumann boundary conditions, we have that $v^{\ep_n}_\theta$ is also zero on the boundary of $\Theta$. In particular,  we find 
\[
v^{\ep_n}_\theta=u^{\ep_n}_\theta=0 \text{ at }(x_{\ep_n},\theta_{\ep_n},t_{\ep_n}).
\]
 Since $u^{\ep_n}-v^{\ep_n}$ has a local minimum at $(x_{\ep_n},\theta_{\ep_n},t_{\ep_n})$ and the previous line holds, we also deduce that  $v^{\ep_n}_{\theta \theta} \leq u^{\ep_n}_{\theta \theta}$ holds at $(x_{\ep_n},\theta_{\ep_n},t_{\ep_n})$, as desired.

Using (\ref{txderivs}), (\ref{thetaderivs}), and that $u^{\ep_n}$ satisfies (\ref{eqn:uep}), we obtain, at $(x_{\ep_n},\theta_{\ep_n},t_{\ep_n})$,
\begin{equation}
\label{vepgeq}
v^{\ep_n}_t - \ep_n \theta_{\ep_n} v^{\ep_n}_{xx} -\theta_{\ep_n} (v^{\ep_n}_x)^2-\frac{\alpha}{\ep_n}v^{\ep_n}_{\theta \theta}  - \alpha\left(\frac{v^{\ep_n}_\theta}{\ep_n}\right)^2 -r(1-\rho^{\ep_n}) \geq 0 . 
\end{equation}
According to Lemma \ref{lem:for main prop}, $v$ satisfies (\ref{eqn for vep}) in $B_s(x_0,t_0)\times \Theta$. 
Since $(x_{\ep_n},t_{\ep_n})\rightarrow (x_0,t_0)$, we have   $(x_{\ep_n},t_{\ep_n})\in B_s(x_0,t_0)$ for $n$ large enough. Hence,  at the point $(x_{\ep_n},\theta_{\ep_n},t_{\ep_n})$, we have,
\[
v^{\ep_n}_t - {\ep_n} \theta_{\ep_n} v^{\ep_n}_{xx}- \theta_{\ep_n} (v^{\ep_n}_x)^2 -\frac{\alpha}{{\ep_n}}v^{\ep_n}_{\theta \theta}  - \frac{\alpha}{{\ep_n}^2} (v^{\ep_n}_\theta)^2 - r \leq -a/2.
\]
Subtracting the previous line from (\ref{vepgeq}) yields,
\[
r\rho^{\ep_n}(x_{\ep_n},t_{\ep_n}) \geq a/2.
\]
But this contradicts (\ref{rho tiny}). We have reached the desired contradiction and conclude that $\underline{u}$ is indeed a supersolution of (\ref{HJ}).
\end{proof}

The proof of Lemma \ref{lem:for main prop} is a simple computation. We include it for the sake of completeness.
\begin{proof}[Proof of Lemma \ref{lem:for main prop}]
We observe 
\[
v^\ep_\theta = \ep \frac{Q_\theta}{Q}; \   \  v^\ep _{\theta\theta} = \ep \frac{Q Q_{\theta \theta} - Q_\theta^2}{Q^2},
\]
so that we have, 
\begin{align*}
v^\ep_t &- \ep \theta v^\ep_{xx} - \theta (v^\ep_x)^2 -\frac{\alpha}{\ep}v^\ep_{\theta \theta}  - \frac{\alpha}{\ep^2} (v^{\ep}_\theta)^2 - r(1 -\rho^\ep) =\\
&=v_t - \ep \theta v_{xx} - \theta (v_x)^2 - \alpha \frac{Q Q_{\theta \theta} - Q_\theta^2}{Q^2} - \alpha \frac{Q_\theta^2}{Q^2}  - r(1-\rho^\ep) \\
&= v_t - \ep \theta v_{xx} - \theta (v_x)^2 - \alpha \frac{Q_{\theta \theta} }{Q}  - r(1- \rho^\ep).
\end{align*}
Thus we have, for some $s$ small, and for all $(x,\theta,t)\in B_s(x_0,t_0)\times\Theta$ and all $\ep$ small enough,
\begin{align*}
v^\ep_t (x,\theta,t)&- \ep \theta v^\ep_{xx} (x,\theta,t)- \theta (v^\ep_x)^2 (x,\theta,t)-\frac{\alpha}{\ep}v^\ep_{\theta \theta}(x,\theta,t)  - \frac{\alpha}{\ep^2} (v^{\ep}_\theta(x,\theta,t))^2 - r(1- \rho^\ep(x,t)) =\\
&\leq v_t(x_0,t_0)  - \theta (v_x)^2(x_0,t_0) - \alpha \frac{Q_{\theta \theta} (\theta) }{Q(\theta) } - r(1- \rho^\ep(x,t))   + a/2\\
& = -a/2 + H(\partial_x v(x_0,t_0))  - \theta (v_x)^2(x_0,t_0) - \alpha \frac{Q_{\theta \theta}(\theta)  }{Q(\theta) }  -  r(1- \rho^\ep(x,t))  \\
&= -a/2 + H( \lambda)  - \theta \lambda^2 - \alpha \frac{Q_{\theta \theta}(\theta)  }{Q(\theta) }  - r(1- \rho^\ep(x,t)) \\
&= -a/2 +r\rho^\ep(x,t),
\end{align*}
where the last equality follows since $Q$ satisfies (\ref{spectral}).
\end{proof}

\subsubsection{Proof that $\bar{u}$ is a subsolution of (\ref{HJ})}
\begin{proof}[Proof that $\bar{u}$ is a subsolution of (\ref{HJ})]
This proof is less involved than that for $\underline{u}$. According to (\ref{sign}) we have $\overline{u}\leq 0$. 

Next, let us suppose that $\overline{u}-v$ has a local maximum at $(x_0,t_0)$. We proceed as in the proof for $\underline{u}$:  for contradiction, we assume  
\[
v_t(x_0,t_0)-H(\partial_x v(x_0,t_0)) = a, 
\]
for some $a>0$. We set $\lambda =\partial_x v(x_0,t_0)$ and take $Q(\theta)$  to be the solution of the spectral problem (\ref{spectral}) corresponding to this $\lambda$. 
As in the previous proof, we define $v^\ep$ by (\ref{def vep}) 
and find that for some $s$ small,  for all $(x,t, \theta)\in B_s(x_0,t_0)\times\Theta$, and for all $\ep$ small enough,
\begin{equation}
\label{eqn:vepsuper}
v^\ep_t (x,\theta,t)- \ep \theta v^\ep_{xx} (x,\theta,t)- \theta (v^\ep_x)^2 (x,\theta,t)-\frac{\alpha}{\ep}v^\ep_{\theta \theta}(x,\theta,t)  - \frac{\alpha}{\ep^2} (v^{\ep}_\theta(x,\theta,t))^2 - r \geq 
  a/2 .
\end{equation}
According to Lemma \ref{lem:u*} there exists a subsequence  $\{\ep_n\}$ and points $(x_{\ep_n}, \theta_{\ep_n}, t_{\ep_n})$  such that
$u^{\ep_n} -v^{\ep_n}$ has a local maximum at $(x_{\ep_n}, \theta_{\ep_n}, t_{\ep_n})$, and $(x_{\ep_n},  t_{\ep_n})\rightarrow (x_0,t_0)$. Since $u^{\ep_n}$ satisfies (\ref{eqn:uep}), we obtain, at $(x_{\ep_n}, \theta_{\ep_n}, t_{\ep_n})$ ,
\begin{equation}
\label{eqn:vepsuper2}
 v^{\ep_n}_t - {\ep_n} \theta_{\ep_n} v^{\ep_n}_{xx} - \theta_{\ep_n} (v^{\ep_n}_x)^2 -\frac{\alpha}{\ep_n}v^{\ep_n}_{\theta \theta}  - \frac{\alpha}{\ep_n^2} (v^{\ep_n}_\theta)^2 - r(1- \rho^{\ep_n})\leq 0.
\end{equation}
(The previous line holds even if $\theta_{\ep_n}$ is a boundary point of $\Theta$, by an argument similar to that in the previous part of the proof of this Proposition). 
For $n$ large enough, $(x_{\ep_n}, t_{\ep_n})\in B_s(x_0,t_0)$, so both (\ref{eqn:vepsuper}) and (\ref{eqn:vepsuper2}) hold at $(x_{\ep_n}, \theta_{\ep_n}, t_{\ep_n})$. Subtracting (\ref{eqn:vepsuper}) from (\ref{eqn:vepsuper2}) yields
\begin{equation}
\label{rhoimpossible}
r\rho^{\ep_n}(x_{\ep_n},t_{\ep_n})\leq  -a/2 ,
\end{equation}
which is impossible since $\rho^{\ep_n}\geq0$.
\end{proof}

\subsubsection{Proof of the auxillary lemma}
Now we present the proof of Lemma \ref{lem:u*}.  
\begin{proof}[Proof of Lemma \ref{lem:u*}]
We will give the proof of the case that $\underline{u}-v$ has a local minimum at $(x_0,t_0)$. The proof of the other case is similar. 

Without loss of generality we assume that $\underline{u}-v$ has a strict local minimum at $(x_0,t_0)$ in $B_{r}(x_0,t_0)$ for some $r>0$. Let $(x_\ep, \theta_\ep, t_\ep)$ be a local minimum of $u^{\ep_n} - v^{\ep_n}$ in $B_{r/2}(x_0,t_0)\times \Theta$.

By the definition of $\underline{u}(x_0,t_0)$, there exists a sequence $(\bar{x}_\ep, \bar{\theta}_{\ep}, \bar{t}_\ep)$ with $(\bar{x}_\ep,  \bar{t}_\ep) \rightarrow (x_0,t_0)$ and such that $u^{\ep}(\bar{x}_\ep, \bar{\theta}_{\ep}, \bar{t}_\ep) \rightarrow \underline{u}(x_0,t_0)$.

We proceed with the proof of item (\ref{itemlimit}) of the lemma. Let $\{(x_{\ep_n},\theta_{\ep_n}, t_{\ep_n})\}_{n=0}^{\infty}$ be any subsequence of $(x_\ep,  t_\ep)$ with $(x_{\ep_n},t_{\ep_n})\rightarrow (y,s)\in \bar{B}_{r/2}(x_0,t_0)$. For $n$ large enough we have $(\bar{x}_\ep, \bar{t}_\ep)\in B_{r/2}(x_0,t_0)$. Since $(x_{\ep_n},\theta_{\ep_n}, t_{\ep_n})$ is a local minimum of $u^{\ep_n} - v^{\ep_n}$ on $B_{r/2}(x_0,t_0)\times \Theta$, we obtain,
\begin{equation}
 \label{eq:uepveplocmin}
u^{\ep_n}(x_{\ep_n},\theta_{\ep_n}, t_{\ep_n}) - v^{\ep_n}(x_{\ep_n},\theta_{\ep_n}, t_{\ep_n}) \leq u^{\ep_n}(\bar{x}_{\ep_n},\bar{\theta}_{\ep_n}, \bar{t}_{\ep_n}) - v^{\ep_n}(\bar{x}_{\ep_n},\bar{\theta}_{\ep_n}, \bar{t}_{\ep_n}).
\end{equation}
We take $\limsup_{n\rightarrow \infty}$ of both sides of (\ref{eq:uepveplocmin}).  
The definition of $(\bar{x}_\ep, \bar{\theta}_{\ep}, \bar{t}_\ep)$  implies that the first term on the right-hand side converges to $\underline{u}(x_0,t_0)$. In addition, since $v$ is continuous and the sequences $(x_{\ep_n},t_{\ep_n})$ and $(\bar{x}_\ep,  \bar{t}_\ep)$ converge to $(y,s)$ and $(x_0,t_0)$, respectively, we find,
\begin{equation}
\label{afterlimsup}
\limsup_{n\rightarrow \infty}(u^{\ep_n}(x_{\ep_n},\theta_{\ep_n}, t_{\ep_n})) - v(y,s) \leq \underline{u}(x_0,t_0)-v(x_0,t_0).
\end{equation}
Since $(x_{\ep_n},t_{\ep_n})\rightarrow (y,s)$, the definition of $\underline{u}(y,s)$ implies,
\begin{equation*}
\limsup_{n\rightarrow \infty}(u^{\ep_n}(x_{\ep_n},\theta_{\ep_n}, t_{\ep_n})) \geq \liminf_{n\rightarrow \infty}(u^{\ep_n}(x_{\ep_n},\theta_{\ep_n}, t_{\ep_n})) \geq \underline{u}(y,s).
\end{equation*}
We use the previous line to estimate the left-hand side of (\ref{afterlimsup}) from below and obtain  $\underline{u}(y,s)-v(y,s) \leq \underline{u}(x_0,t_0)-v(x_0,t_0)$. Since $(x_0,t_0)$ is a strict local maximum, we see $(y,s)=(x_0,t_0)$. This completes the proof of items (\ref{itemlimit}) and (\ref{itemmin}). 

Next let us take $\liminf_{n\rightarrow \infty}$ of both sides of (\ref{eq:uepveplocmin}). Since we now know $(x_\ep,  t_\ep)\rightarrow (x_0,t_0)$ and $(\bar{x}_\ep, \bar{t}_\ep)\rightarrow (x_0,t_0)$, the terms with $v$ are equal. In addition, the definition of $(\bar{x}_\ep, \bar{\theta}_{\ep}, \bar{t}_\ep)$  implies that the first term on the right-hand side converges to $\underline{u}(x_0,t_0)$. Thus we find
\[
 \liminf_{n\rightarrow \infty}(u^{\ep_n}(x_{\ep_n},\theta_{\ep_n}, t_{\ep_n}))  \leq \underline{u}(x_0,t_0).
\]
Since $(x_\ep,  t_\ep)\rightarrow (x_0,t_0)$, the definition of $\underline{u}(x_0,t_0)$ implies that equality holds in the above. This completes the proof of the lemma.
\end{proof}

\subsection{Proof of Part (II) of Proposition \ref{main result}} 
\label{sec:at time 0}
In the previous subsection we showed that $\bar{u}$ and $\underline{u}$ are a supersolution and subsolution, respectively, of (\ref{HJ}) on $\rr\times (0,\infty)$.  In this section we study the behavior of $\bar{u}$ and $\underline{u}$ at time $t=0$.

\begin{proof}[Proof that $\bar{u}$ satisfies (\ref{eqn:us at time 0})]
First we show that if $x\in J$ then $\bar{u}(x,0) = 0$. Indeed, if $x\in J$ then there exists a point $\theta^*\in \Theta$ such that $n_0(x,\theta^*) = a >0$. Hence
\[
\bar{u}(x,0)\geq \lim_{\ep\rightarrow 0} \ep \ln n^\ep(x,\theta^*) =  \lim_{\ep\rightarrow 0} \ep \ln a =0.
\]
And, according to our observation (\ref{sign}) we have $\bar{u}(x,0)\leq 0$, so we obtain $\bar{u}(x,0)= 0$. 

In the remainder of the proof we will analyze the behavior of $\bar{u}(x,0)$ for $x\notin \bar{J}$. 
To this end, we observe that, by the definition of $J$,
\[
\rr\setminus J = \{x: \   n_0(x,\theta)=0 \text{ for all }\theta\in \Theta.\}
\]
We fix a constant $\mu>0$ and a cutoff function $\zeta\in C^{\infty}(\rr)$ that satisfies
\begin{equation}
\label{cutoff}
\begin{cases}
\zeta = 0 \text{ on } \bar{J}, \    \zeta>0 \text{ on }\rr\setminus{\bar{J}}&\\
0\leq \zeta\leq 1.&
\end{cases}
\end{equation}

\textbf{First step:} We claim that $\bar{u}$ is a viscosity subsolution of
\begin{equation}
\label{var ineq at 0}
\min\{\bar{u}_t-H(\bar{u}_x), \bar{u}+\mu \zeta\}\leq 0 \text{ on } \rr\times \{0\},
\end{equation}
by which we mean if $\bar{u}-v$ has a local maximum at $(x_0,0)$ for some test function $v$, then either 
\begin{equation}
\label{ugeq0}
\bar{u}(x_0,0)  \leq  -\mu \zeta(x_0)
\end{equation}
or
\begin{equation}
\label{eqnvx0}
v_t- H(v_x) \leq 0 \text{ at }(x_0,0).
\end{equation}
For $x\in J$, we have $\bar{u}(x_0,0)=0$ and $\mu \zeta(x_0) =0$, so we find (\ref{ugeq0}) holds.  

Now let us suppose $x_0\in \rr\setminus \bar{J}$, (\ref{ugeq0}) doesn't hold, and $\bar{u} - v$ has a local maximum at $(x_0,0)$ for some smooth $v$.  We proceed as in the proof of Part (I) of Proposition \ref{main result}: let us assume for contradiction that (\ref{eqnvx0}) also does not holds, so that 
\begin{equation}
\label{noteqnvx0}
v_t- H(v_x) =a>0 \text{ at }(x_0,0)
\end{equation}
for some $a>0$. We set $\lambda = \partial_xv(x_0,0)$ and take $Q(\theta)$ to be the solution of the spectral problem (\ref{spectral}) corresponding to this $\lambda$.  We  define $v^\ep$ by (\ref{def vep}), and find that 
there exist points $(x_{\ep_n}, \theta_{\ep_n}, t_{\ep_n})$  such that  $(x_{\ep_n},  t_{\ep_n})\rightarrow (x_0,t_0)$  as $n\rightarrow \infty$, 
$u^{\ep_n} -v^{\ep_n}$ has a local maximum at $(x_{\ep_n}, \theta_{\ep_n}, t_{\ep_n})$,  and $\bar{u}(x_0,0)=\lim u^{\ep_n}(x_{\ep_n}, \theta_{\ep_n}, t_{\ep_n})$. We claim  
\begin{equation}
 \label{tep}
t_{\ep_n}>0 \text{ for all $n$ large enough}.
\end{equation}
If (\ref{tep}) holds, then $u^{\ep_n}$ satisfies the equation (\ref{eqn:uep}) at $(x_{\ep_n}, \theta_{\ep_n}, t_{\ep_n})$. Hence the argument in Part (I) of the proof of Proposition \ref{main result} applies in this situation as well and leads to a contradiction (namely, we will find that (\ref{eqn:vepsuper}), (\ref{eqn:vepsuper2}) both hold, which again yields (\ref{rhoimpossible}), and the latter cannot hold). Thus, once we show (\ref{tep}), we will find that $v$ satisfies (\ref{eqnvx0}) and so we will have established (\ref{var ineq at 0}).

We proceed by contradiction and assume that (\ref{tep}) does not hold. Thus, there exists a subsequence of $\{\ep_n\}$, also denoted $\{\ep_n\}$, with $t_{\ep_n}=0$. Since $x_0\in \rr\setminus \bar{J}$, there exists $s>0$ such that  $B_s(x_0)\subset (\rr\setminus \bar{J})$. We have that for all $n$ large enough, $x_{\ep_n}\in B_s(x_0)$ and hence $n_0(x_{\ep_n}, \theta_{\ep_n})=0$ for all $n$ large enough. Since, by definition, we have $u^\ep(\cdot, \cdot, 0)=\ep \ln n_0(\cdot, \cdot)$, we obtain,
\[
\bar{u}(x_0,0)=\lim_{n\rightarrow \infty} u^{\ep_n}(x_{\ep_n}, \theta_{\ep_n}, t_{\ep_n}) =-\infty.
\]
But we had assumed (\ref{ugeq0}) doesn't hold, so that $\bar{u}(x_0,0)>-\mu \zeta(x_0,0)>-\infty$. We have reached the desired contradiction, hence (\ref{tep}) must hold, and so we conclude that $\bar{u}$ is a viscosity subsolution of (\ref{var ineq at 0}).

\textbf{Second step:} We will now use that $\bar{u}$ is a viscosity subsolution of (\ref{var ineq at 0}) to prove $\bar{u} (x, 0)= -\infty$ on $\rr\setminus \bar{J}$. To this end, let us fix any $x_0\in \rr\setminus \bar{J}$ and assume for contradiction  
\begin{equation}
\label{uM}
\bar{u}(x_0,0)=-M>-\infty.
\end{equation}
For $\delta>0$, let us define the test functions
\[
v^\delta(x,t) = \frac{|x-x_0|^2}{\delta} + \nu(\delta) t,
\]
where we use $\nu(\delta)$ to  denote
\begin{equation}
\label{choicemu}
\nu(\delta)=\frac{8\theta_M M}{\delta}+r.
\end{equation}
Since $\bar{u}$ is upper-semicontinuous, there exist $(x_\delta,t_\delta)$  such that $\bar{u}-v^\delta$ has a maximum at $(x_\delta,t_\delta)$ in $\rr\times [0,\infty)$. In particular, we have 
\[
 (\bar{u}-v^\delta)(x_\delta,t_\delta) \geq (\bar{u}-v^\delta)(x_0,0)=\bar{u}(x_0,0) =-M,
\]
where the  equalities follow since $v^\delta(x_0,0)=0$ and from (\ref{uM}). We now use  that $\bar{u}\leq 0$ 
 and the definition of $v^\delta$ to estimate the left-hand side of the previous line from above by $-\frac{|x_\delta-x_0|^2}{\delta}$. Thus we obtain an estimate on the distance between $x_\delta$ and $x_0$:
\begin{equation}
\label{estxdelx0}
\frac{|x_\delta-x_0|^2}{\delta} \leq M \text{ for all }\delta.
\end{equation}

We will now establish the inequality
\begin{equation}
\label{estonquantity}
\nu(\delta) - H\left(\frac{2(x_\delta-x_0)}{\delta}\right)>0 \text{ for all }\delta.
\end{equation}
Indeed, according to Proposition \ref{prop:spectral}, we have $-H(\lambda) \geq -\lambda^2 \theta_M-r$ for all $\lambda$. We apply this with $\lambda =\frac{2(x_\delta-x_0)}{\delta}$ and obtain,
\[
-H\left(\frac{2(x_\delta-x_0)}{\delta}\right) \geq -\frac{4(x_\delta-x_0)^2}{\delta^2}\theta_M -r \geq 
-\frac{4\theta_M M}{\delta}-r,
\]
where the second inequality follows from (\ref{estxdelx0}).  Let us add $\nu(\delta)$ to both sides of the previous line. The left-hand side becomes exactly the left-hand side of (\ref{estonquantity}). The right-hand side becomes  $\frac{4\theta_M M}{\delta}$, due to our choice of $\nu$ in (\ref{choicemu}). Thus we find (\ref{estonquantity}) holds.

We now recall that $\bar{u}-v^\delta$ has a maximum at $(x_\delta,t_\delta)$ in $\rr\times [0,\infty)$. 
Let us suppose that  $t_\delta>0$ for some $\delta$. According to Part (I) of Proposition \ref{main result}, $\bar{u}$ is a subsolution of (\ref{HJ}) in $\rr\times (0,\infty)$, which implies
\begin{equation}
\label{eqn v at xdelta}
\nu(\delta) - H\left(\frac{2(x_\delta-x_0)}{\delta}\right) \leq 0.
\end{equation}
But this is impossible, as we have just established (\ref{estonquantity}).  Therefore, we must have $t_\delta=0$ for all $\delta$. But we also know that 
$\bar{u}$ is a subsolution of (\ref{var ineq at 0}) on $\rr\times \{0\}$. Therefore, we have
\[
\min\left\{ \nu(\delta) - H\left(\frac{2(x_\delta-x_0)}{\delta}\right) , \bar{u}(x_\delta,0) + \mu \zeta(x_\delta)\right\} \leq 0.
\]
But, again according to (\ref{estonquantity}),  we have that the first term inside the $\min$ must be strictly positive. Therefore, in order for the previous line to hold, the second term inside the $\min$ must be non-positive. Thus we have,
\begin{equation}
\label{bumuzeta}
\bar{u}(x_\delta,0) \leq -\mu \zeta(x_\delta).
\end{equation}
Since  $(x_\delta, 0)$ is a local maximum of $\bar{u}-v^\delta$, we have
\[
\bar{u}(x_0,0)-v^\delta(x_0,0) \leq (\bar{u}-v^\delta)(x_\delta,0) .
\]
Because  $-M=\bar{u}(x_0,0)$ and $v^\delta(x_0,0)=0$, we have that the left-hand side of the previous line is exactly $-M$.  In addition, we use that $v^\delta$ is non-negative and the estimate (\ref{bumuzeta}) to bound the right-hand side of the previous line from above. We find,
\[
 -M \leq -\mu \zeta(x_\delta).
\]
Since $x_\delta \rightarrow x_0$ as $\delta\rightarrow 0$ and $\zeta$ is continuous, we obtain $-M\leq -\mu \zeta(x_0)$, which is impossible since $\zeta(x_0)>0$ and $\mu$ is arbitrary. We have obtained the desired contradiction and thus  (\ref{uM}) cannot hold. We conclude $\bar{u}(x_0,0)=-\infty$, and hence the proof is complete.
\end{proof}

\begin{proof}[Proof that $\underline{u}$ satisfies (\ref{eqn:u under at 0})]
This proof is similar to the one for $\bar{u}$. First we show that if $x_0\in \rr\setminus K$, then $\underline{u}(x_0,0)=-\infty$.  Indeed, since $x_0\in \rr\setminus K$, there exists $\theta_0$ with $n_0(x_0,\theta_0)=0$. Therefore,
\[
 \underline{u}(x_0,0)\leq \liminf_{\ep\rightarrow 0} u^\ep(x_0,\theta_0,0) = \liminf_{\ep\rightarrow 0} \ep \ln (n_0(x_0,\theta_0))=-\infty.
\]

We will now prove that $\underline{u}(x_0,0)=0$ for $x_0\in K$. 

\textbf{First step:} We prove that $\underline{u}$ is a supersolution of
\begin{equation}
 \label{var ineqn underu}
 \max\{\underline{u},\underline{u}_t - H(\underline{u}_x)\}\geq 0 \text{ on }K\times \{0\}.
\end{equation}
To this end, let us suppose  $x_0\in K$, 
\begin{equation}
\label{uux0}
\underline{u}(x_0,0)<0,
\end{equation}
and $\underline{u}-v$ has a local minimum at $(x_0,0)$ for some test function $v$. We proceed as in the proof of Part (I) of Proposition \ref{main result}: we  define $v^\ep$ by (\ref{def vep}), and find that
there exist points $(x_{\ep_n}, \theta_{\ep_n}, t_{\ep_n})$  such that $(x_{\ep_n},  t_{\ep_n})\rightarrow (x_0,t_0)$ as $n\rightarrow \infty$, 
$u^{\ep_n} -v^{\ep_n}$ has a local minimum at $(x_{\ep_n}, \theta_{\ep_n}, t_{\ep_n})$,   and $\underline{u}(x_0,0)=\lim_{n\rightarrow \infty} u^{\ep_n}(x_{\ep_n}, \theta_{\ep_n}, t_{\ep_n})$. We claim  
\begin{equation}
 \label{tep for uunder}
t_{\ep_n}>0 \text{ for all $n$ large enough}.
\end{equation}
If (\ref{tep}) holds, then the argument in the proof of Part (I) of Proposition \ref{main result} applies in this situation as well. Thus, once we show (\ref{tep for uunder}), we will find that $v$ satisfies $v_t+v_xc(v_x)\geq 0$ at $(x_0,0)$ and so we will have established that (\ref{var ineqn underu}) in the viscosity sense.

We proceed by contradiction and assume that (\ref{tep for uunder}) does not hold. Thus, there exists a subsequence, also denoted $\{\ep_n\}$, with $t_{\ep_n}=0$. Since $x_0\in K$, there exists $s>0$ such that  $B_s(x_0)\subset K$. Since $x_n\rightarrow x_0$, we have that for all $n$ large enough, $x_{\ep_n}\in B_s(x_0)$. Therefore, there exists $m>0$ so that $n_0(x_{\ep_n},\theta_{\ep_n})\geq m$ for all $n$ large enough. Hence we find,
\[
 \underline{u}(x_0,0)=\lim_{n\rightarrow \infty} u^{\ep_n}(x_{\ep_n}, \theta_{\ep_n}, 0) =\lim_{n\rightarrow \infty} \ep_n \ln(n_0(x_{\ep_n},\theta_{\ep_n})) \geq \lim_{n\rightarrow \infty} \ep_n \ln(m) =0,
\]
where the first equality follows from the definition of the sequence $(x_{\ep_n}, \theta_{\ep_n}, t_{\ep_n})$. But the previous line contradicts our assumption (\ref{uux0}). Therefore, (\ref{tep for uunder}) must hold, and we conclude that $\underline{u}$ is a viscosity supersolution of (\ref{var ineqn underu}).

\textbf{Second step:} Let us fix $x_0\in K$. We will prove $\underline{u}(x_0,0)\geq 0$, which, together with (\ref{sign}), implies $\underline{u}(x_0,0)=0$. Let us suppose for contradiction 
\begin{equation}
\label{undux0}
\underline{u}(x_0,0)<0.
\end{equation} 
We point out that  $\underline{u}(x_0,0)$  is finite. Indeed, if $\underline{u}(x_0,0)=-\infty$ then $\underline{u}-v$ has a minimum at $(x_0,0)$ for all $v$, and since we know $\underline{u}$ is a supersolution of (\ref{var ineqn underu}), we find $v_t-H(v_x)\geq 0$ at $(x_0,0)$ for all $v$, which is of course impossible.

Since $\underline{u}$ is lower-semicontinuous and finite at $(x_0,0)$, there exists a neighborhood $Q$ of $(x_0,0)$ and some finite $M>0$  such that if  $(x,t)\in Q$, then $\underline{u}(x,t)>-M$.

We define, for $\delta>0$, the test functions
\[
v^\delta(x,t) = -\frac{|x-x_0|^2}{\delta} - \nu(\delta) t,
\]
where we define $\nu$ by (\ref{choicemu}) as in the previous proof. 
There exists a sequence $(x_\delta,t_\delta)\rightarrow (x_0,0)$ such that $\underline{u}-v^\delta$ has a local minimum at  $(x_\delta,t_\delta)\in Q$. We find, as in the previous proof, 
that the upper bound   (\ref{estxdelx0}) on $|x_\delta-x_0|^2$  holds here as well. We use (\ref{estxdelx0}), the definition of $\nu(\delta)$  and the properties of $H$ given in Proposition \ref{prop:spectral} to obtain, for all $\delta$,
\begin{equation}
\label{negubar}
-\nu(\delta) + H\left(-\frac{2(x_\delta-x_0)}{\delta}\right) < 0 \text{ for all }\delta.
\end{equation}
Let us suppose $t_\delta>0$. Since $\underline{u}-v^\delta$ has a minimum at $(x_\delta,t_\delta)$ and, according to Part (I) of Proposition \ref{main result}, $\underline{u}$ is a supersolution of (\ref{HJ}) in $\rr\times (0,\infty)$, we find that $\geq$ must hold in \ref{negubar}). But this is impossible, since we have already established (\ref{negubar}). 
 Therefore, we find $t_\delta=0$ for all $\delta$. According to (\ref{var ineqn underu}), we have that $\underline{u}$ is a supersolution of (\ref{var ineqn underu}) on $K\times \{0\}$, so we find
 \[
 \max\left\{ \underline{u}(x_\delta,0), -\mu + H\left(-\frac{2(x_\delta-x_0)}{\delta}\right)\right\} \geq 0.
 \] 
 According to (\ref{negubar}) we see that the second term in the $\max$ is strictly negative. Hence we obtain  $\underline{u}(x_\delta,0)\geq 0$ for all $\delta$. Together with  $v^\delta(x_0,0)=0$, the fact that  $(x_\delta, 0)$ is a minimum of $\underline{u}-v^\delta$, and $v^\delta\leq 0$, this implies,
\[
 \underline{u}(x_0,0) =  \underline{u}(x_0,0)- v^\delta(x_0,0) \geq \underline{u}(x_\delta,0)- v^\delta(x_\delta,0)\geq 0.
\]
But this contradicts our assumption (\ref{undux0}). Thus we find $\underline{u}(x_0,t)\geq 0$ and the proof is complete.
\end{proof}

\section{Proof of the main result}
\label{sec:result on n}
In this short section we  use the results that we've established in the rest of the paper to give the proof of our main result, Theorem \ref{result on n}. The arguments are similar to those in the proofs of \cite[Theorem 1.1]{EvansSouganidis} and   \cite[Theorem 1]{BarlesEvansSouganidis}.

\begin{proof}[Proof of Theorem \ref{result on n}]
Since (\ref{HJ}) satisfies the comparison principle (see Proposition \ref{prop:prelem}), Proposition \ref{main result} and the definitions of $u_J$ and $u_K$  imply $\bar{u}\leq u_J$ and $\underline{u}\geq u_K$ on $\rr\times (0,\infty)$.

Let us fix some $\bar{B}_s(x_0,t_0)\subset \{u_J<0\}$. Since $u_J$ is continuous, there exists $a>0$ such that  $u_J< -a$ for all $(x,t)\in \bar{B}_s(x_0,t_0)$. Since $\bar{u}\leq u_J$, we have $\bar{u}(x,t) <-a$ for all $(x,t)\in \bar{B}_s(x_0,t_0)$ as well. Therefore, $u^\ep(x,\theta,t)\leq -\frac{a}{2}$ for all $(x,t)\in \bar{B}_s(x_0,t_0)$, for all $\theta\in \Theta$ and for all $\ep$ small enough. Thus, for all $(x,t)\in \bar{B}_s(x_0,t_0)$ and for all $\theta$, we have
\[
0\leq n^\ep(x,\theta,t)=e^{\frac{u^\ep(x,\theta,t)}{\ep}}\leq e^{\frac{-a}{2\ep}},
\]
so we find that $n^\ep(x,\theta,t)\rightarrow 0$ uniformly at an exponential rate on $\bar{B}_s(x_0,t_0)\times \Theta$.

Now let us suppose $(x_0,t_0)$ is a point in the interior of $\{u_K=0\}$, so that $B_s(x_0,t_0)\subset \{u_K=0\}$ for some $s>0$. According to (\ref{sign}), we have $\underline{u}\leq 0$ on $\rr\times(0,\infty)$.  Since $\underline{u}\geq u_K$, we have $\underline{u}\geq 0$ on $B_s(x_0,t_0)$, so we see $\underline{u}\equiv 0$ on $B_s(x_0,t_0)$. We define the test function
\[
 \phi(x,t)=-(x-x_0)^2-(t-t_0)^2.
\]
Since $\underline{u}\equiv 0$ on $B_s(x_0,t_0)$, we find that $\underline{u}-\phi$ has a local minimum at $(x_0,t_0)$. Therefore there exists a sequence $(x_\ep,\theta_\ep, t_\ep)$ with $(x_\ep,t_\ep)\rightarrow (x_0,t_0)$ such that $u^\ep-\phi$ has a local minimum at $(x_\ep,\theta_\ep, t_\ep)$. Since $u^\ep$ is satisfies (\ref{eqn:uep}) in the viscosity sense, we find, at $(x_\ep,\theta_\ep, t_\ep)$,
\[
 \phi_t\geq\ep \theta_\ep \phi_{xx}  +\theta (\phi_x)^2 +r(1-\rho^\ep).
\]
Using the definition of $\phi$ and rearranging yields,
\[
 r\rho^\ep(x_\ep,t_\ep) \geq r -2\ep\theta_\ep +2\theta (x_\ep-x_0)^2 + 2(t_\ep-t_0).
\]
Taking the limit  $\ep\rightarrow 0$  of the previous line yields,
\[
 \lim_{\ep\rightarrow 0}\rho^\ep(x_\ep,t_\ep) \geq 1.
\]
Recalling the definition of $\limsup^*$ completes the proof.
\end{proof}

\subsection{Proofs of the corollaries}
\label{subsect:pf of cors}
Corollary \ref{corc*informal} follows easily from our main result, Theorem \ref{result on n}, and the following lemma, which is essentially a restatement of a result of Majda and Sougandis \cite{MS}. Before stating the lemma, we introduce one more bit of notation. Given a bounded interval $\Omega\subset \rr$, we use $g_\Omega:\rr\rightarrow\rr$ to denote a $C^1$  bounded function that is positive on $\Omega$, negative on $\rr\setminus \bar{\Omega}$, with a maximum at some $\bar{x}\in\Omega$,  strictly increasing to the left of $\bar{x}$ and strictly decreasing to the right of $\bar{x}$. 

\begin{lem}
\label{lemfornewcor}
Assume $\Omega\subset \rr$ is a bounded interval. Let $w_\Omega$ be the unique viscosity solution of 
\begin{equation}
\label{eqn:wOm}
\begin{cases}
\partial_t w_\Omega = c^* |\partial_x w_\Omega| \text{ on }\rr\times (0,\infty),\\
w_\Omega(x,0)=g_\Omega (x) \text{ on } \rr.
\end{cases}
\end{equation}
Let $u_\Omega$ be given by Lemma 1.1. Then:
\begin{enumerate}
\item \label{itemMS} $\{u_\Omega <0\}\supseteq \{w_\Omega <0\}$ and $\{u_\Omega =0\}\supseteq\{w_\Omega \geq 0\}$; and,
\item \label{itemchar} $\{x: w_\Omega(x,t) <0\} = \{x: d(x,\Omega) > c^*t\}$ and  $\{x: w_\Omega(x,t) > 0\} = \{x: d(x,\Omega) < c^*t\}$.  
\end{enumerate}
\end{lem}

In \cite{MS}, the authors analyze the properties of the front $\Gamma_t = \bdry \{ x: u (x,t) <0\}$, where $u$ solves an equation of the form (\ref{HJ}). It turns out that, for a general Hamiltonian $H$ that depends on space and time, this front may not be very regular. However,  our situation is rather simple and no degeneracy appears -- indeed, item (\ref{itemMS})  says  that our front is geometric. 

\begin{proof}[Proof of Lemma \ref{lemfornewcor}]
Item (\ref{itemMS}) is a special case of Propositions 2.3 and 2.4 of \cite{MS}. Indeed, what we denote $\Omega$, $u_\Omega$, $w_\Omega$ are denoted $G_0$, $Z$, $u$ in \cite{MS}. And, if we use $\tilde{H}$ to denote the Hamiltonian $H$ of \cite{MS}, reserving $H$ itself for our Hamiltonian, then we have the correspondence $H(p) \equiv \tilde{H}(p)+r$ between the two notations. In addition, the velocity field $V$ of \cite{MS} is $0$ in our situation, so the hypotheses of both propositions are satisfied. Finally, we may use (\ref{Hc*}) to find that the general nonlinearity $F$ in line (2.5) of \cite{MS} is, in our situation, simply given by $F(p)=|p|c^*$. (For the benefit of the reader, we also point out two typos in section 2.2 of \cite{MS}. First, ``(1.6)" should read ``(1.9)" throughout that section. Second, ``$u(x,t)=0$" should read ``$u(x,t)\geq 0$" in the conclusion of Proposition 2.3.)

 The proof of item (\ref{itemchar}) is based on the method of characteristics.  The characteristics for equations of the form $w_t+\tilde{F}(Dw)=0$ are analyzed in detail in Barles' \cite[page 17]{Barles85}. Taking $\tilde{F}(p)=- |p|c^*$ in order to study (\ref{eqn:wOm})  yields that the characteristic  emanating from $x_0$ is 
 \[
 x_{x_0}(t)=x_0-t \sign(g'(x_0)) c^*,
 \]
(where $\sign(\alpha) = 1$ if $\alpha>0$, $-1$ if $\alpha<0$, and $0$ if $\alpha=0$),  and,
 \[
 w_\Omega( x_{x_0}(t), t) = g(x_0) \text{ for all }t>0.
 \]
Let us use the properties of $g_\Omega$ to  continue our analysis. We recall that $g_\Omega$ has a maximum at $\bar{x}$, is strictly increasing to the left of $\bar{x}$ and strictly decreasing to the right. Hence, if $x_0$ is to the left of  $\bar{x}$, then the characteristic has constant slope $-1/c^*$ in the $(x,t)$ plane; while, if $x_0$ is to the right of $\bar{x}$, then the slope is $1/c^*$. We remark that  the characteristics never cross. And, $w_\Omega$ is constant along these characteristics, so,  in particular, the sign of the initial condition is preserved along them.

Now, let $(x,t)$ be a point with $d(x,\Omega) > c^*t$. We shall show $w_\Omega(x,t) <0$.  Let us suppose $x$ lies to the right of $\bar{x}$ (the other case is similar). Then $(x,t)$ must  lie on a characteristic with slope $1/c^*$, so that $(x,t)= (x_0+tc^*, t)$ for some $x_0$. We compute,
\[
d(x_0,\Omega) \geq d(x_0+tc^*, \Omega)- tc^*= d(x, \Omega)- tc^* >0,
\]
where the equality follows since $(x,t)= (x_0+tc^*, t)$, and the second inequality since $d(x,\Omega) > c^*t$.  Since $x_0$ lies outside $\Omega$, $g_\Omega$ is negative there. Because $w$ is constant along characterisics, we conclude $w_\Omega(x,t) = g_\Omega(x_0) <0$. We omit the similar arguments that establish the remainder of the proposition.
\end{proof}

We now proceed with:
\begin{proof}[Proof of Corollary \ref{corc*informal}]
Let $u_J$ and  $u_K$ be as in Lemma \ref{lem intro} and let $w_J$ and $w_K$ be as in Lemma \ref{lemfornewcor}. Combining the two parts of Lemma \ref{lem intro} yields,
\[
\{(x,t): d(x, J) > c^*t\} \subseteq \{(x,t): u_J(x,t) <0\} \text{ and } \{(x,t): d(x, K) < c^*t\} \subseteq \{(x,t): u_K(x,t) =0\}. 
\]
On the other hand, Theorem \ref{result on n} implies,
\begin{itemize}
\item If  $u_J(x,t) <0$, then $\displaystyle \lim_{\ep\rightarrow 0} n^\ep(x,\theta, t) =0$ for all $\theta\in \Theta$.
\item If $(x,t)$ is in the interior of $\{u_K(x,t) =0\}$, then  $\displaystyle \limsup_{\ep\rightarrow 0}{ }^* \rho^\ep(x,t)\geq 1$.
\end{itemize}
Combining this with the previous line yields the  Corollary. 
\end{proof}

We now use Corollary \ref{corc*informal} and the bounds (\ref{boundsonc*}) on $c^*$ to present:
\begin{proof}[Proof of Corollary \ref{cor:onn}]
To establish the first statement of the corollary, let  $(x,t)$ be such that $\dist(x, J)>2t\sqrt{\theta_M r}$. According to  the second inequality in (\ref{boundsonc*}), we also have $\dist(x, J)>tc^*$. Using Corollary \ref{corc*informal} completes the proof of the first statement; the proof of the second is analogous.
\end{proof}

\appendix
\section{}
\label{Appendix}

In this appendix, we state an existence result and  a comparison principle for (\ref{HJ}) with infinite initial data (Proposition \ref{prop:prelem}).  These were established in \cite[Section 4]{BarlesEvansSouganidis} but not explicitely stated there, and we could not locate another reference in the literature. Because of this, we  carefully explain how to obtain them from \cite[Section 4]{BarlesEvansSouganidis}.  
In addition to \cite{EvansSouganidis, BarlesEvansSouganidis}, we also refer the reader to Crandall, Lions and Souganidis \cite{CLS} for more about Hamilton-Jacobi equations with infinite initial data. 
\begin{prop}
\label{prop:prelem}
There exists  a viscosity solution $w$ of (\ref{HJ}) on $\rr\times (0, \infty)$ with with initial data (\ref{infinite}). Moreover, suppose $\underline{w}$ and $\bar{w}$ are, respectively,  a viscosity subsolution and a viscosity supersolution of (\ref{HJ}) on $\rr\times (0, \infty)$, and  both with  the initial data (\ref{infinite}). Then we have
\begin{equation}
\label{ws}
\underline{w}\leq w\leq \bar{w}.
\end{equation}
In particular, the solution to (\ref{HJ}) with initial data (\ref{infinite}) is unique.
\end{prop}
\begin{proof}
We explain why Proposition \ref{prop:prelem} is exactly what was proven in \cite[Section 4]{BarlesEvansSouganidis}. Since both we and  \cite{BarlesEvansSouganidis} use $H$ as notation for the Hamiltonian, for the purposes of this proof we use  $\tilde{H}$ to denote the Hamiltonian of  \cite{BarlesEvansSouganidis}. 

There is a difference in sign between our paper and \cite[Section 4]{BarlesEvansSouganidis}, which we now address. We have that $-w$, $-\underline{w}$ and $-\bar{w}$ are, respectively, a viscosity solution, supersolution, and subsolution of 
\[
\min\{ u, u_t +H(-u_x)\}=0
\]
with initial data
\begin{equation*}
\begin{cases}
0 &\text{ for } x\in \Omega\\
\infty &\text{ for } x \in \rr\setminus \bar{\Omega}.
\end{cases}
\end{equation*}
Thus we see that, if we take $\tilde{H}(p)=H(-p)$ and $G_0=\Omega$, then we are exactly in the situation of \cite[Section 4]{BarlesEvansSouganidis}. We may take $I=-w$, $v_*=-\underline{w}$, and $v^*=-\bar{w}$. Thus we've established the existence of a solution. Plus, the conclusion of the comparison of the arguments in \cite[Section 4]{BarlesEvansSouganidis} (specifically, lines (4.2) and (4.5)) is
\[
v^*\leq I\leq v_*.
\]
Translating back to our notation, we see (\ref{ws}) holds. 
\end{proof}

\section*{Acknowledgements} 
The author thanks her thesis advisor,  Takis Souganidis, for  his guidance and encouragement. 
The author is grateful to Vincent Calvez and Sepideh Mirrahimi for   reading earlier drafts of this paper very thoroughly. Their remarks were invaluable. 
In addition, the author thanks Benoit Perthame for a stimulating discussion that led to the developement of Corollary \ref{corc*informal}. 
The author also thanks the anonymous referees. Their comments have greatly helped improve the exposition   and  raised several interesting questions; for example,  Remarks \ref{remark: relationship} and \ref{remark:choice of nonlinearity} were motivated by their reports.

\end{document}